\documentclass[12pt]{amsart}
\usepackage{amsmath,amssymb,amsfonts,amsthm,amstext,graphicx,color,enumerate}
\usepackage{hyperref,cleveref,url,mathtools,mathdots}
\usepackage{tikz}
\usepackage[numbers,sort&compress]{natbib}
\usepackage{verbatim}
\usetikzlibrary{arrows}

\title{Finitely dependent coloring}
\date{14 March 2014 (revised 6 April 2015)}

\newtheorem{thm}{Theorem}
\newtheorem*{thm*}{Theorem}
\newtheorem{prop}[thm]{Proposition}
\newtheorem{lemma}[thm]{Lemma}
\newtheorem{cor}[thm]{Corollary}

\crefname{thm}{Theorem}{Theorems}
\crefname{lemma}{Lemma}{Lemmas}
\crefname{prop}{Proposition}{Propositions}
\crefname{cor}{Corollary}{Corollaries}
\crefname{section}{Section}{Sections}
\crefname{figure}{Figure}{Figures}

\newcommand{\E}{\mathbb E}
\renewcommand{\P}{\mathbb P}

\newcommand{\Z}{\mathbb Z}
\newcommand{\R}{\mathbb R}
\newcommand{\ind}{\mathbf{1}}
\DeclareMathOperator{\DD}{DD}
\DeclareMathOperator{\trace}{trace}

\newcommand{\p}{{+}}
\newcommand{\m}{{-}}
\renewcommand{\hat}{\widehat}
\newcommand{\simm}{\stackrel m\sim}
\newcommand{\ph}{p_{\text{\rm h}}}

\newcounter{mycount}
\newenvironment{ilist}{\begin{list}{\rm(\roman{mycount})}%
   {\usecounter{mycount}\leftmargin=9mm\itemindent=0pt\labelwidth=9mm\itemsep 0pt}}{\end{list}}

\newcommand{\df}[1]{\textbf{\boldmath #1}}

\begin{document}

\author[Holroyd]{Alexander E.\ Holroyd}
\address{Alexander E.\ Holroyd, Microsoft Research,
1 Microsoft Way, Redmond, WA 98052, USA}
\email{holroyd at microsoft.com}
\urladdr{\url{http://research.microsoft.com/~holroyd/}}

\author[Liggett]{Thomas M.\ Liggett}
\address{Thomas M.\ Liggett, Department of Mathematics,
University of California, Los Angeles, CA 90095, USA}
\email{tml at math.ucla.edu}
\urladdr{\url{http://www.math.ucla.edu/~tml/}}



\vspace{-1cm} \maketitle

\vspace{-3mm} {\centering\em Dedicated to Oded Schramm, 10 December 1961 -- 1
September 2008\\}

\begin{abstract}
We prove that proper coloring distinguishes between
block-factors and finitely dependent stationary
processes.  A stochastic process is finitely dependent if
variables at sufficiently well-separated locations are
independent; it is a block-factor if it can be expressed
as an equivariant finite-range function of independent
variables.  The problem of finding non-block-factor
finitely dependent processes dates back to 1965.  The
first published example appeared in 1993, and we provide
arguably the first natural examples.  More precisely,
Schramm proved in 2008 that no stationary $1$-dependent
$3$-coloring of the integers exists, and conjectured that
no stationary $k$-dependent $q$-coloring exists for any
$k$ and $q$. We disprove this by constructing a
$1$-dependent $4$-coloring and a $2$-dependent
$3$-coloring, thus resolving the question for all $k$ and
$q$.

Our construction is canonical and natural, yet very
different from all previous schemes. In its pure form it
yields precisely the two finitely dependent colorings
mentioned above, and no others.  The processes provide
unexpected connections between extremal cases of the
Lov\'asz local lemma and descent and peak sets of random
permutations. Neither coloring can be expressed as a
block-factor, nor as a function of a finite-state Markov
chain; indeed, no stationary finitely dependent coloring
can be so expressed.  We deduce extensions involving $d$
dimensions and shifts of finite type; in fact, any
non-degenerate shift of finite type also distinguishes
between block-factors and finitely dependent processes.
\end{abstract}

\section{Introduction}

Central to probability and ergodic theory is the notion of
mixing in various forms.  A stochastic process is a family
of random variables indexed by a metric space, and mixing
means that variables at distant locations are approximately
independent.  The strongest and simplest mixing condition
is finite dependence, which states that subsets of
variables are independent provided they are at least some
fixed distance apart. Despite the simplicity of the
definition, finite dependence turns out to be rather
subtle.  Finitely dependent processes arise in the context
of classical limit theorems
\cite{hoeffding-robbins,ibragimov-linnik-2,heinrich-1990,janson-deg},
renormalization of statistical physics models
\cite{obrien,lss}, and the Lov\'{a}sz local lemma
\cite{lovasz,alon-spencer}, a fundamental tool of
probabilistic combinatorics.

A key problem, originating from work of Ibragimov and
Linnik in 1965 \cite{ibragimov-linnik,ibragimov-linnik-2},
has been to understand the relationship between finite
dependence and block-factors.  A block-factor is a process
that can be expressed as a function of an underlying family
of independent random variables, where the function is
finite-range and commutes with the action of a transitive
symmetry group. It is clear that a block-factor is finitely
dependent; it is natural to ask about the converse
implication.

This question retains its interest and subtlety even in the
simplest setting of stochastic processes indexed by the
integer line.  (We return to more general settings later.)
We say that a stochastic process $X=(X_i)_{i\in\Z}$ is
\df{$k$-dependent} if the random sequences
$(\ldots,X_{i-2},X_{i-1})$ and $(X_{i+k},X_{i+k+1},\ldots)$
are independent of each other, for each $i\in\Z$; if $X$ is
$k$-dependent for some integer $k$ then it is \df{finitely
dependent}.  A process $X$ is \df{stationary} if
$(X_i)_{i\in \Z}$ and $(X_{i+1})_{i\in \Z}$ are equal in
law.  A process $X$ is an \df{$r$-block-factor} (of an
i.i.d.\ process) if for some i.i.d.\ $(U_i)_{i\in\Z}$ and
some measurable $f$ we have
$X_i=f(U_{i+1},U_{i+2},\ldots,U_{i+r})$ for each $i$. (The
random variables $U_i$ can be assumed uniform on $[0,1]$
without loss of generality.)

An $r$-block-factor is clearly stationary and
$(r-1)$-dependent.  Ibragimov and Linnik
\cite{ibragimov-linnik,ibragimov-linnik-2} proved in 1965
that the converse implication holds for Gaussian processes,
and claimed without proof that it is false in general. This
question was explicitly stated as open by G\"otze and Hipp
\cite{gotze-hipp} and Janson \cite{janson-84}.  It was not
resolved until 1989, when Aaronson, Gilat, Keane and de
Valk \cite{aaronson-gilat-keane-devalk} gave a family of
$1$-dependent processes that are not $2$-block-factors.
This construction is indirect and algebraic, and the
authors asked for more natural examples. This question and
the surrounding issues have been taken up by a number of
authors \cite{aaronson-gilat-keane,burton-goulet-meester,
matus-1996,matus-1998,haiman,devalk-1989,
gandolfi-keane-devalk,ruschendorf-devalk,
janson-83,janson-deg,devalk,hoeffding-robbins,lss,broman},
and various further examples have been constructed.
Highlights include an explicit $1$-dependent ($5$-state)
Markov chain that is not a $2$-block factor
\cite{aaronson-gilat-keane}, a (hidden-Markov)
$1$-dependent process that is not an $r$-block-factor for
any $r$ (Burton, Goulet and Meester, 1993
\cite{burton-goulet-meester}), and a ``perturbable''
example showing that $2$-block-factors are not dense in the
set of $1$-dependent Markov chains \cite{matus-1996}.

The constructions mentioned above are intricate, subtle and
counterintuitive, but the resulting examples have the
appearance of technical ones specifically constructed for
the purpose. For instance, Borodin, Diaconis and Fulton
\cite{adding} remarked in 2010: `it appears that most
``natural'' one-dependent processes are two-block factors'.
This issue has practical implications: several authors
\cite{gotze-hipp,janson-84,heinrich-1990} have been forced
to assume a block-factor representation as an additional
assumption in the study of finitely dependent processes: if
natural finitely dependent processes are block factors,
then there is little to be lost by such an assumption.

In this article we provide arguably the first genuinely
natural finitely dependent stationary process that is not a
block-factor.  Moreover, we establish something much
stronger, which runs entirely counter to the above ideas
about natural processes. Suppose that we impose any fixed
system of local constraints on a stochastic process.
(Formally, we require the process to belong almost surely
to a shift of finite type.)  Provided the constraints
satisfy certain simple non-degeneracy conditions, we show
that they can be satisfied by a stationary finitely
dependent process, but not by any block-factor. The latter
negative statement follows from ideas of Ramsey theory --
our main contribution is the former positive statement.
Underlying this is a remarkable new stochastic process that
is natural and canonical, yet apparently quite different
from all previously studied classes of stochastic
processes. It has many surprising properties that hint at a
deeper theory. In particular, certain marginal projections
provide unexpected links between known processes involving
descent and peak sets of random permutations, Dyck words,
and extremal cases of the Lov\'{a}sz local lemma.

Proper coloring is a canonical choice of local constraint,
which turns out to be the key to the general case. We call
a stochastic process $X=(X_i)_{i\in \Z}$ a
\df{$q$-coloring} (of $\Z$) if each $X_i$ takes values in
$\{1,\ldots,q\}$, and almost surely we have $X_i\neq
X_{i+1}$ for all $i\in\Z$.  For which $k$ and $q$ does
there exist a stationary $k$-dependent $q$-coloring of
$\Z$?  This question arose from discussions between Itai
Benjamini, Alexander Holroyd and Benjamin Weiss in early
2008.  In addition to its implications in relation to block
factors, it is a formulation of the very natural question:
do local constraints demand global organization?  It can
also be seen a question about spontaneous
symmetry-breaking. Oded Schramm proved a negative answer in
the first non-trivial case: there is no stationary
$1$-dependent $3$-coloring. The proof appears in
\cite{hsw}; we will give different proof, which provides
some further information. Schramm conjectured that no
stationary $k$-dependent $q$-coloring exists for any $k$
and $q$. We disprove this.

\enlargethispage*{1cm}
\begin{thm}\label{main}
There exist a stationary $1$-dependent $4$-coloring of
$\Z$, and a stationary $2$-dependent $3$-coloring of $\Z$.
\end{thm}
On the other hand, we have the following.
\begin{prop}\label{block-factor}
No $r$-block-factor $q$-coloring exists, for any $r$ and
$q$.
\end{prop}
\cref{main,block-factor} together provide perhaps the
cleanest answer one could hope for to the question raised
by Ibragimov and Linnik:

\begin{trivlist}\em\item
Coloring can be done by a
stationary $1$-dependent process, but not by a
block-factor.
\end{trivlist}

Moreover, since it is easily seen that no stationary
finitely dependent $2$-coloring exists, Schramm's
impossibility result and \cref{main} together provide a
complete answer to the above question about $k$-dependent
$q$-colorings.  In fact, there is a canonical construction
that gives precisely the two required cases
$(k,q)=(1,4),(2,3)$ in \cref{main}, and no others. To our
knowledge, \cref{main} also provides the first stationary
finitely dependent non-block-factor that is symmetric under
permutations of the symbols, and the first stationary
$1$-dependent process that is not hidden-Markov. (See below
for details.)

We do not claim \cref{block-factor} as new, although it
does not appear to be particularly well known in this form.
An essentially equivalent result appears in \cite{naor} (in
a stronger, quantitative form, stated in rather different
terms motivated by applications in distributed computing,
and building on earlier work in \cite{linial}). Further
extensions and applications appear in \cite{hsw,alon}. For
the reader's convenience we provide a simple proof of
\cref{block-factor}.

Given the prominence of Markov chains in the literature on
finitely dependent processes, it is natural to ask whether
our colorings are Markov. They are not, and much more can
be said. We call a stationary process $X$
\df{hidden-Markov} if there exists a stationary Markov
chain $M=(M_i)_{i\in\Z}$ on a finite state space, and a
deterministic function $f$, such that $X_i=f(M_i)$ for all
$i$. (In contrast with the definition of block-factors,
here finiteness of the state space is important: if we were
to allow an uncountable state space then {\em any}
stationary process $X$ could be represented this way, by
taking $M_i=(\ldots,X_{i-1},X_i)$.)  Note that
hidden-Markov processes include $m$-step Markov processes,
as well as Gibbs measures with local interactions. The
following is a previously unpublished result of Schramm, of
which we present a proof.

\begin{prop}[Schramm]\label{hidden}
No hidden-Markov finitely dependent $q$-coloring exists,
for any $q$.
\end{prop}

In particular, our $4$-coloring provides a partial answer
to a question of de Valk \cite[Problem~8]{devalk}, who
asked whether every $1$-dependent process is a function of
a Markov chain: the answer is no for finite-state chains.
(The case of countable state spaces remains open).

As mentioned earlier, the colorings of \cref{main} have
many remarkable properties, which hint at some deeper
structure. We strongly believe that the stationary
$1$-dependent $4$-coloring is unique.  The next result
gives some of these properties, and also provides a small
step towards uniqueness. Let $\ind[\,\cdot\,]$ denote an
indicator function.

\begin{thm}\label{properties}
The stationary $1$-dependent $4$-coloring $X$ and
$2$-dependent $3$-coloring $Y$ of \cref{main} can be chosen
to have the following additional properties.
\begin{ilist}
  \item The processes are reversible, and
  symmetric under permutations of the colors, i.e.\
$X$ is equal in law to $(X_{-i})_{i\in\Z}$, and to
$(\sigma(X_i))_{i\in \Z}$ for any
  $\sigma\in S_4$, and similarly for $Y$ and $\sigma\in S_3$.
  \item The process $(\ind[X_i=1])_{i\in\Z}$ is equal in law to
      $(\ind[B_i>B_{i+1}])_{i\in\Z}$, where $(B_i)_{i\in \Z}$ are i.i.d.\
      taking values $0,1$ with equal probabilities.
 \item The process $(\ind[X_i\in\{1,2\}])_{i\in\Z}$ is equal
     in law to $(\ind[U_i>U_{i+1}])_{i\in\Z}$, where $(U_i)_{i\in\Z}$ are
     i.i.d.\ uniform on $[0,1]$.
 \item The process $(\ind[Y_i=1])_{i\in\Z}$ is equal
     in law to $(\ind[U_{i-1}{<}U_i{>}U_{i+1}])_{i\in\Z}$, where $(U_i)_{i\in\Z}$ are
     i.i.d.\ uniform on $[0,1]$.
 \item The law of $(Y_1,\ldots,Y_n)$ is the conditional law of $(X_1,\ldots,X_n)$
  given that $X_i\neq 4$ for $i=1,\ldots,n$.
\end{ilist}
Every stationary $1$-dependent $4$-coloring $X$ satisfies
\textrm{(ii)}.
\end{thm}

The processes in (ii)--(iv) above are evidently
block-factors, notwithstanding \cref{block-factor}.  Many
of these properties are mysterious.  It is not clear why
conditioning a $1$-dependent $4$-coloring to have no $4$'s
should be expected to give a $2$-dependent process, as in
(v). We have no simple explanation for the striking
similarity between (iii) and (iv) (even bearing in mind
(v)).  It appears difficult to think of {\em any} processes
satisfying the properties above, or even certain subsets of
them. For example, we know of no other ergodic process $X$
that satisfies (i) and (ii), nor that satisfies the
analogue of (iii) for every $2$-element subset of
$\{1,2,3,4\}$.  It appears plausible that some such sets of
properties may uniquely characterize the processes.

The processes in (ii)--(iv) have been studied extensively
in other settings; (ii) is the unique extremal case of the
Lov\'asz local lemma (see \cite{scott-sokal,scott-sokal2}
and the discussion below), and (iii) and (iv) correspond to
the descent sets and peak sets of random permutations (see
e.g.\ \cite{billey-burdzy-sagan} and references therein).
The colorings $X$ and $Y$ can be seen as couplings of
multiple copies of these processes (with special
properties).

We will prove \cref{main} by giving expressions for
cylinder probabilities (i.e.\ for the probability that
$(X_1,\ldots,X_n)$ takes any given value) in terms of a
certain combinatorial structure.  The expressions are
simple but mysterious, and seem {\em a priori} very hard to
guess.  In the case of the $4$-coloring, we will prove that
the expression is equal to a very different (and more
complicated) expression (an alternating sum of numbers of
linear extensions of certain posets), which is useful for
deducing certain properties including
\cref{properties}(iii) above.  The equality of the two
expressions also implies many interesting new combinatorial
identities.  We in fact started with the more complicated
expression (which was guessed by considering the
constraints imposed on a $4$-coloring by $1$-dependence),
but were unable to prove its nonnegativity directly.  We
were led to the simple expression by searching for
recursions satisfied by the complicated one.

We now consider generalizations to higher dimensions, and
to general systems of local constraints (as mentioned
earlier). Firstly, let $G=(V,E)$ be a graph. A stochastic
process $X=(X_v)_{v\in V}$ indexed by the vertices is
called a \df{$q$-coloring} if each $X_v$ takes values in
$\{1,\ldots,q\}$ and almost surely $X_u\neq X_v$ whenever
$u$ and $v$ are neighbors.  It is \df{$k$-dependent} if its
restrictions to two subsets of $V$ are independent whenever
the subsets are at graph-distance greater than $k$ from
each other. The hypercubic lattice is the graph with vertex
set $\Z^d$ and an edge between $u$ and $v$ whenever
$\|u-v\|_1=1$; the graph itself is also denoted $\Z^d$. A
process on $\Z^d$ is \df{stationary} if it is invariant in
law under all translations of $\Z^d$.
\begin{cor}\label{zd}  Let $d\geq 2$.
 There exist integers $q=q(d)$ and $k=k(d)$ such
that:
  \begin{ilist}
    \item there exists a stationary $1$-dependent $q$-coloring of $\Z^d$;
    \item there exists a stationary $k$-dependent $4$-coloring of $\Z^d$.
  \end{ilist}
\end{cor}
No stationary $k$-dependent $q$-coloring of $\Z^d$ was
previously known to exist for any $k,q,d$.  The proof of
\cref{zd} yields explicit upper bounds on $q(d)$ and
$k(d)$, but we do not expect them to be close to optimal.
In particular we can take $q(d)=4^d$ in (i). (See
\cref{bounds} below for some lower bounds.) Both assertions
are consequences of \cref{main}; (i) is straightforward to
deduce, while (ii) uses results of Holroyd, Schramm and
Wilson \cite{hsw} that were developed for the study of
finitary factor colorings. While the colorings of \cref{zd}
are stationary under translations, we do not know how to
make them invariant under all isomeries of $\Z^d$.  By
another result in \cite{hsw}, the $4$ colors in (ii) cannot
be reduced to $3$ for any $d\geq 2$.

To describe our second extension we generalize from proper
coloring to arbitrary local constraints.  Write
$[q]:=\{1,\ldots,q\}$.  A \df{shift of finite type} on $\Z$
is a (deterministic) set of sequences $S\subseteq[q]^\Z$
characterized by an integer $m$ and a set $W\subseteq[q]^m$
of allowed local patterns as follows:
$$S=S(q,m,W):=\bigl\{ x\in [q]^\Z:
(x_{i+1},\ldots,x_{i+m})\in W\;\forall i\in \Z\bigr\}.$$
 For $w\in W$, let $T(w)$ be the set of times at which the
pattern $w$ can recur, i.e.\ the set of $t\geq 1$ for which
there exists $x\in S$ with $(x_1,\ldots,x_m)$ and
$(x_{t+1},\ldots,x_{t+m})$ both equal to $w$.  We call the
shift of finite type \df{non-lattice} if there exists $w\in
W$ for which $T(w)$ has greatest common divisor $1$. For
example, the set of all deterministic proper $q$-colorings
of $\Z$ is a shift of finite type, and is non-lattice if
and only if $q\geq 3$.  The following is again a
consequence of \cref{main} together with results from
\cite{hsw}.
\begin{cor}\label{sft}
Let $S$ be a non-lattice shift of finite type on $\Z$.
There exists an integer $k$ (depending on $S$) and a
stationary $k$-dependent process $X$ such that the random
sequence $X$ belongs to $S$ almost surely.
\end{cor}

The following is a straightforward consequence of
\cref{block-factor}, proved in \cite{hsw}. Let $S$ be a
shift of finite type on $\Z$ that does not contain any
constant sequence. Then there is no block-factor that
belongs a.s.\ to $S$. (In fact, under the non-lattice
condition, it is shown in \cite{hsw} that there is a
\textit{finitary} factor of an i.i.d.\ process, with tower
function decay of its coding radius, that belongs a.s.\ to
$S$, and that this decay rate cannot be improved).
Combining this with \cref{sft} provides, as promised, an
even more striking answer to the Ibragimov-Linnik question:
\begin{trivlist}\em\item
{\em Any} non-lattice shift of finite type on $\Z$ that
contains no constant sequence serves to distinguish between
block-factors and stationary finitely dependent processes.
\end{trivlist}

Returning to coloring, for any graph $G$ and any $k$ and
$q$ one can ask whether there exists a $k$-dependent
$q$-coloring that is invariant in law under some given
group of automorphisms.  The following concept leads to
negative answers in some cases. A \df{hard-core} process on
$G$ is a process $J=(J_v)_{v\in V}$ such that each $J_v$
takes values in $\{0,1\}$, and almost surely we do not have
$J_u=J_v=1$ for adjacent vertices $u,v$.  If $X$ is a
$q$-coloring of $G$ then $J_v:=\ind[X_v=a]$ defines a
hard-core process for any given color $a\in[q]$.  If $X$ is
$k$-dependent then so is $J$. We define the critical point
\begin{multline*}
\ph=\ph(G):= \sup \\
\{p: \exists\text{ a $1$-dependent hard-core process $J$ with
}\P(J_v=1)=p\;\forall v\}.
\end{multline*}
Intriguingly, it turns out that for each $p\leq \ph$ there
is a {\em unique} $1$-dependent hard-core process with all
one-vertex marginals $\P(J_v=1)$ equal to $p$. Moreover,
$\ph$ has alternative interpretations involving complex
zeros of the partition function of the standard hard-core
model (or lattice gas) of statistical physics, and in terms
of boundary cases of the Lov\'asz local lemma.  See
\cref{sec-hardcore} and \cite{scott-sokal,scott-sokal2} for
details.

Suppose that there exists a $1$-dependent $q$-coloring $X$
of $G$ in which the colors $(X_v)_{v\in V}$ are identically
distributed.  (This last condition holds in particular if
the process is invariant in law under a transitive group of
automorphisms). Then the above remarks imply
\begin{equation}\label{hard-core-bound}
q \geq \frac{1}{\ph},
\end{equation}
so upper bounds on $\ph$ yield lower bounds on the number
of colors needed. We illustrate the method by proving the
following.
\begin{prop}\ \label{bounds}
Suppose that there exists a $1$-dependent $q$-coloring $X$
of $G$ with $(X_v)_{v\in V}$ identically distributed.
\begin{ilist}
\item For $G=\Z^d$ we have $q\geq (d+1)^{d+1}/d^d$, and
    moreover $q\geq 9$ for $d=2$, and $q\geq 12$ for
    $d=3$.
\item For $G=T_\Delta$, the infinite $\Delta$-regular
    tree, $q\geq \Delta^\Delta/(\Delta-1)^{\Delta-1}$.
\end{ilist}
\end{prop}

We do not know the minimum number of colors needed for a
stationary $1$-dependent coloring of $\Z^d$ for any $d\geq
2$.  On the tree $T_\Delta$, one may use \cref{main} to
construct $1$-dependent colorings that are invariant in law
under certain transitive groups of automorphisms, but again
we do not know the minimum number of colors, nor whether
fully automorphism-invariant versions exist.

It is a remarkable fact that the bound
\eqref{hard-core-bound} is tight on $\Z$: we have
$\ph(\Z)=1/4$, yet there exists a stationary $1$-dependent
$4$-coloring.  In other words, it is possible to couple $4$
copies of the critical $1$-dependent hard-core process in
such a way that their supports partition $\Z$, while the
entire process retains stationarity and $1$-dependence.

One can interpret $k$-dependent processes via the language
of functional analysis (see also \cite{devalk}).  The
following is a consequence of \cref{main}.

\begin{cor}\label{analysis}
Let $(k,q)=(1,4) \text{ or }(2,3)$. There exists a real
separable Hilbert space $U$ and a bounded linear operator
$R:U\to U$ with the following properties. The image $R^n U$
is one-dimensional for all $n> k$.
 There is a decomposition $U=U_1 + \cdots + U_q$ into mutually orthogonal
closed linear subspaces, such that for each $i$, the image
$R \,U_i$ is contained in the closed linear span of
$\{U_j:j\neq i\}$.
\end{cor}

So far as we know, \cref{analysis} is new.  Schramm
conjectured in 2008 (motivated by colorings) that such $U$
and $R$ cannot exist for any $k$ and $q$ (even with the
$U_j$ merely linearly independent, and without the
separability restriction).  A space $U$ satisfying the
conditions of the corollary cannot be finite-dimensional,
and by Lidskii's theorem (see e.g.\ \cite[Chapter
30]{lax}), $R$ cannot be of trace class. A {\em complex}
Hilbert space example has been suggested by Fedja Nazarov
and Serguei Denissov (personal communication).

We now give a complete probabilistic description of our two
colorings of $\Z$, which is astonishingly simple. (However,
it is not at all obvious that it works; we will prove this
in the next two sections.) See \cref{pic}. Let
$Z=(Z_1,\ldots,Z_n)$ be a sequence of i.i.d.\ random
variables taking values $1,2,\ldots,q$ with equal
probabilities.  Let $\sigma$ be an independent uniformly
random permutation of $1,\ldots,n$, which we interpret as
meaning that the symbol $Z_i$ arrives at time $\sigma(i)$.
Let $E$ be the event that, for every time $t=1,\ldots,n$,
the subsequence of $Z$ formed by those symbols that arrived
up to time $t$ (ordered as in the original sequence $Z$)
forms a proper coloring (i.e.\ no two consecutive elements
in the subsequence are equal). Then for $q=4$ or $q=3$, the
conditional law of $Z$ given $E$ equals the law of
$(X_1,\ldots,X_n)$, where $X$ is, respectively, the
$4$-coloring or the $3$-coloring of \cref{main}.
\begin{figure}
\centering
\begin{tikzpicture}[scale=0.7]
\definecolor{c1}{rgb}{0,0,0.6};
\definecolor{c2}{rgb}{0.6,0,1};
\definecolor{c3}{rgb}{1,0,0};
\definecolor{c4}{rgb}{0,0.8,0};
\begin{scope}[line width=3pt,->]
 \draw[c3] (1,4) -- (1,0);
 \draw[c3] (3,3) -- (3,0);
 \draw[c3] (6,5) -- (6,0);
 \draw[c1] (2,6) -- (2,0);
 \draw[c2] (4,1) -- (4,0);
 \draw[c4] (5,2) -- (5,0);
\end{scope}
\begin{scope}[line width=4pt,fill=white]
 \filldraw[draw=c3,text=c3]
 (1,4) circle (12pt) node{\sffamily\bfseries 3}
 (3,3) circle (12pt) node{\sffamily\bfseries 3}
 (6,5) circle (12pt) node{\sffamily\bfseries 3};
 \filldraw[draw=c1,text=c1]
 (2,6) circle (12pt) node{\sffamily\bfseries 1};
 \filldraw[draw=c2,text=c2]
 (4,1) circle (12pt) node{\sffamily\bfseries 2};
 \filldraw[draw=c4,text=c4]
 (5,2) circle (12pt) node{\sffamily\bfseries 4};
\end{scope}
\draw (7.5,3.6) node[rotate=90] {time};
 \foreach \x in {1,2,3,4,5,6}
  \draw (\x,0) node[anchor=north] {$Z_{\x}$};
 \foreach \x in {1,2,3,4,5,6}
  \draw (6.5,7-\x) node[anchor=west] {${\x}$};
 \draw[line width=1pt,double distance=3pt,implies-implies]
  (3.65,3) -- (5.8,3);
\end{tikzpicture}
\caption{Construction of the process:
random colors arrive at random times.
In this case the coloring $(Z_1,\ldots,Z_6)$ is rejected at time $4$,
because $Z_3$ and $Z_6$ are both red (color $3$),
and they arrive before the intervening points $Z_4$
and $Z_5$.}\label{pic}
\end{figure}
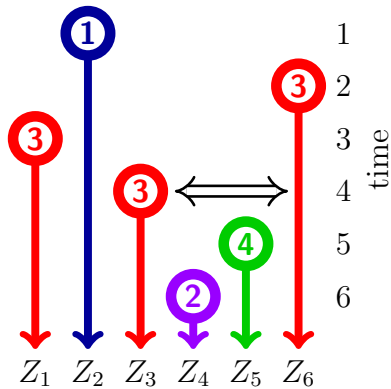

We emphasize that the cases $q=3,4$ in the above
description are very special.  For $q=2$ or $q\geq 5$, the
resulting process is not $k$-dependent for any $k$.

In a follow-up article \cite{hl-qcol} by the current
authors, we use a more elaborate method inspired by the
construction above to obtain for all $q\geq 5$ a stationary
$1$-dependent $q$-coloring of $\Z$ that is symmetric under
permutations the colors (as in \cref{properties}(i)).
Besides these examples and straightforward embellishments
of them, no other stationary finitely dependent colorings
of $\Z$ are known.

In another article \cite{h-finfin} by one of the current
authors, the above construction is modified to obtain a
probabilistic construction of the $4$-coloring on the whole
of $\Z$.  (More precisely, the process is expressed as a
finitary factor of an i.i.d.\ process; however, the
approach fails for the $3$-coloring). One complication is
that, while the laws of colorings $(X_1,\ldots,X_n)$ are
consistent between different intervals (as required to
obtain an extension to $\Z$), the accompanying random
permutations (after conditioning) are \textit{not}
consistent.

The article \cite{hsw} deals with the closely related issue
of coloring $\Z^d$ by a finitary factor of an i.i.d.\
process; that is, a deterministic function that commutes
with translations in which the color at the origin can
determined from the i.i.d.\ variables within some finite
(but random and perhaps unbounded) radius. Depending on the
number of colors and the dimension, it turns out that the
optimal tail decay of this radius is either a power law or
a tower function.

The relationship between the $4$-coloring and $3$-coloring
is puzzling. Can they be coupled in a natural way (without
conditioning)? Here is one plausible approach that fails.
If $X$ is a $1$-dependent $4$-coloring then we can obtain a
$3$-dependent $3$-coloring $Y$ as a $3$-block-factor of $X$
by eliminating color $4$: take $Y_i$ to be $X_i$ unless
$X_i=4$, in which case
$Y_i:=\min(\{1,2,3\}\setminus\{X_{i-1},X_{i+1}\})$.  It is
natural to try to get a $2$-dependent $3$-coloring as a
$2$-block-factor of $X$, but this is impossible -- this
amounts to the fact that the Kautz graph with vertices
$V=\{(a,b)\in\{1,2,3,4\}^2: a\neq b\}$ and (undirected)
edges $E=\{((a,b),(b,c)):(a,b),(b,c)\in V\}$ is not
3-colorable.

Coloring, finite dependence, and block-factors have
applications in computer science (see e.g.\
\cite{naor,linial}). For example, colors may represent
update schedules or communication frequencies for machines
in a network; adjacent machines are not permitted to
conflict with each other.  Finite dependence implies
privacy or security benefits: an adversary who gains
knowledge of some colors learns nothing about the others,
except within some fixed distance.  A block-factor (or,
more generally, a finitary factor \cite{hsw,h-finfin}) has
the interpretation that colors can be computed by the
machines in a distributed fashion, based on randomness
generated locally together with local communication.

The article is organized as follows.  In \cref{buildings}
we introduce a combinatorial structure on which our
processes are based.  In \cref{colorings} we deduce
\cref{main} and \cref{properties}(i,v).
\cref{block,hilbert,renewal-structure,alternative,extras,sec-hardcore}
can largely be read independently of each other.  In
\cref{block,hilbert} we give proofs of
\cref{block-factor,hidden} respectively, the latter using
the Hilbert space interpretation that also gives
\cref{analysis}.  In \cref{renewal-structure} we prove
\cref{properties}(ii,iv) together with the stronger
assertion that every $1$-dependent $4$-coloring has the
former property, and we give a new proof of Schramm's
result that no $1$-dependent $3$-coloring exists.  In
\cref{alternative} we provide the alternative expression
for the cylinder probabilities, and deduce
\cref{properties}(iii).  \cref{extras} contains the proofs
of \cref{zd,sft}, and in \cref{sec-hardcore} we discuss
hard-core processes and prove \cref{bounds}.  We conclude
the article with a list of open problems.

\section{Buildings}
\label{buildings}

In this section we introduce the combinatorial object on which our
construction is based.  We deduce some striking properties, although the real
magic will happen when we interpret them probabilistically.

A \df{word} is a finite sequence $x=(x_1,x_2,\ldots,x_n)\in\Z^n$, which we
sometimes abbreviate to $x_1x_2\cdots x_n$.  The word $x$ is a \df{proper
coloring} if $x_i\neq x_{i+1}$ for all $1\leq i< n$.  For a word $x\in\Z^n$
and a symbol $a\in\Z$ we denote the concatenation as $xa=(x_1,\ldots,x_n,a)$,
etc. We write $\widehat{x}_i:=x_1\cdots x_{i-1}x_{i+1}\cdots x_n$ for $x$
with the $i$th symbol removed.

Let $S_n$ be the symmetric group of all permutations of $1,\ldots,n$.  Let
$x\in\Z^n$ be a word, and let $\sigma\in S_n$ be a permutation. We interpret
$\sigma$ as meaning that the symbol $x_i$ arrives at time $\sigma(i)$ (and in
position $i$). For $t=1,\ldots,n$ we define
$$x^\sigma_{(t)}:=(x_i: \sigma(i)\leq t),$$
the subsequence of symbols that arrived by time $t$ (ordered as in $x$, {\em
not} ordered by arrival times). So for example if $\sigma=(2,3,1)$ then
$x^\sigma_{(2)}=(x_1,x_3)$. We say that $\sigma$ is a \df{proper building} of
$x$ if $x^\sigma_{(t)}$ is a proper coloring for each $t=1,\dots,n$.  So the
identity permutation is a proper building of the word $121$, but the
permutation $(2,3,1)$ is not. Let $B(x)$ denote the number of proper
buildings of $x$.  The following is the key property.

\begin{lemma}\label{build-rec}
If $x$ is a proper coloring of length $n$ then
$$B(x)=\sum_{i=1}^n B(\widehat{x}_i).$$
\end{lemma}

\begin{proof}
This follows on considering the last arrival $\sigma^{-1}(n)$.  The
permutation $\sigma$ is a proper building of $x$ with $\sigma^{-1}(n)=i$ if
and only if $\hat\sigma_i$ is a proper building of $\hat x_i$.
\end{proof}

We deduce the following identities.  Recall that $[q]:=\{1,\ldots,q\}$.

\begin{prop}\label{build-cons}
Let $q\geq 2$ and $x\in[q]^n$, where $n\geq 0$.  We have
$$\sum_{a\in [q]} B(xa) = \bigl[n(q-2)+q\bigr] B(x).$$
\end{prop}
\begin{prop}\label{build-dep}
Let $x\in[q]^m$ and $y\in[q]^n$, where $m,n\geq 0$.
\begin{align*}
\label{build-dep-eq}
&\text{If $q=4$ then}&\sum_{a\in [q]} B(xay) &= 2 \binom{m+n+2}{m+1} B(x) B(y). \\
&\text{If $q=3$ then}&\sum_{a,b\in[q]} B(xaby)&=2\binom{m+n+4}{m+2} B(x)B(y).
\end{align*}
\end{prop}

The proofs of \cref{build-cons,build-dep} are elementary, and are very
similar to each other.  However, in another respect the two results are very
different: \cref{build-dep} says something special about $q=3,4$ that
apparently has no simple analogue for other $q$.  For example, for $q\neq 4$
the ratio of $\sum_{a\in [q]} B(xay)$ to $B(x)B(y)$ no longer depends only on
the lengths of $x$ and $y$.  Also see \cref{two-point} at the end of this
section.

\begin{cor}\label{build-tot}
Let $q\geq 2$ and $n\geq 1$.  The total number of proper buildings of all
words of length $n$ is
$$\Sigma(q,n):=\sum_{x\in[q]^n} B(x) =
\prod_{k=1}^n \bigl[k(q-2)+2\bigr],$$
 which equals $2^n$, $(n+2)!/2$, and $(n+1)!\,2^n$ in the cases
$q=2,3,4$ respectively.
\end{cor}

\begin{proof}
This is immediate from \cref{build-cons}.    (The last factor in the product
is $(n-1)(q-2)+q=n(q-2)+2$).
\end{proof}

\begin{proof}[Proof of \cref{build-cons}]
We use induction on $n$. The identity is immediate when $n=0$ (so that $x$ is
the empty word and $B(x)=1$). Suppose that $n\geq 1$ and that it holds for
$n-1$.  We can assume that $x$ is a proper coloring, otherwise both sides are
$0$.  By \cref{build-rec},
\begin{equation}\label{cons-sum}
 \sum_{a\in [q]} B(xa) =
 \sum_{a\neq x_n}\biggl[\sum_{i=1}^n B(\widehat{x}_i a)+B(x)\biggr].
\end{equation}
We now consider each of the terms on the right. For $i\leq n-1$ the inductive
hypothesis gives
\begin{align*}
\sum_{a\neq x_n}
B(\widehat{x}_i a)&=\bigl[(n-1)(q-2)+q\bigr]B(\widehat{x}_i),
\intertext{while for the case $i=n$ we have}
\sum_{a\neq x_n} B(\widehat{x}_n a)+B(\hat x_n x_n)
&=\bigl[(n-1)(q-2)+q\bigr]B(\widehat{x}_n).
\end{align*}
Since $\hat x_n x_n=x$, and $\sum_{a\neq x_n}B(x)=(q-1)B(x)$, the right side
of \eqref{cons-sum} therefore becomes
$$\bigl[(n-1)(q-2)+q\bigr]\sum_{i=1}^n B(\widehat{x}_i)+(q-2)B(x),$$
which by \cref{build-rec} equals $[n(q-2)+q]B(x).$
\end{proof}

\begin{proof}[Proof of \cref{build-dep}, case $q=4$]
We use induction. When $n=0$ the identity is precisely \cref{build-cons} with
$q=4$, and the case $m=0$ follows by symmetry. Therefore, suppose that
$m,n\geq 1$, and that the identity holds for all $x$ and $y$ with lengths
totalling less than $m+n$.  Assume that $x$ and $y$ are proper colorings,
otherwise the identity holds trivially.

We consider two cases (and the crucial consequence of the assumption $q=4$
will be that they give identical results).  First suppose $x_m=y_1$, and
without loss of generality suppose both are equal to $1$. \cref{build-rec}
gives
\begin{equation}\label{dep-sum}
 \sum_{a\in [4]} B(xay) =
 \sum_{a\neq 1}\biggl[\sum_{i=1}^m B(\widehat{x}_i ay)
 +B(xy)+\sum_{j=1}^n B(x a\widehat{y}_j)\biggr].
\end{equation}
Considering the first of the three terms on the right, the inductive
hypothesis gives for each $i$,
$$\sum_{a\neq 1} B(\widehat{x}_i ay) = 2\binom{m+n+1}{m}
B(\widehat{x}_i)B(y).$$
 Similar reasoning applies to the third term, while
$B(xy)=0$ since $xy$ is not a proper coloring.  Therefore, using
\cref{build-rec} again, the right side of \eqref{dep-sum} equals
\begin{equation}\label{2-binoms}
2\binom{m+n+1}{m}B(x)B(y)+ 2\binom{m+n+1}{m+1}B(x)B(y),
\end{equation}
which equals the right side of the claimed identity.

For the second case, suppose $x_m\neq y_1$, and say $x_m=1$ and $y_1=2$. Then
\begin{equation}\label{dep-sum-neq}
 \sum_{a\in [4]} B(xay) =
 \sum_{a=3,4}\biggl[\sum_{i=1}^m B(\widehat{x}_i ay)
 +B(xy)+\sum_{j=1}^n B(x a\widehat{y}_j)\biggr].
\end{equation}
For $i\leq m-1$ we have, similarly to the previous case,
\begin{align*}
\sum_{a=3,4}
B(\widehat{x}_i ay) &= 2\binom{m+n+1}{m} B(\widehat{x}_i)B(y).
\intertext{On the other
hand, for $i=m$, the inductive hypothesis gives}
\sum_{a=3,4} B(\widehat{x}_m ay) + B(xy)
&=\sum_{a\neq 2} B(\widehat{x}_m ay)\\
 &= 2\binom{m+n+1}{m} B(\widehat{x}_m)B(y).
\end{align*}
The last of the three terms on the right of \eqref{dep-sum-neq} can be
 treated similarly, and of course the middle term yields $\sum_{a=3,4}
B(xy)=2B(xy)$. (This is the key point where $q=4$ is used -- for general $q$
we would be left with an additional term $(q-4)B(xy)$, which was not present
in the first case above.)  Therefore the right side of \eqref{dep-sum-neq}
equals \eqref{2-binoms}, as in the previous case.
\end{proof}

\begin{proof}[Proof of \cref{build-dep}, case $q=3$]
The proof is similar to the $q=4$ case, and is again by induction. When $m$
or $n$ is $0$, the result follows by applying \cref{build-cons} (twice).
Therefore suppose $m,n\geq 1$ and that the result holds for all smaller
$m+n$.  Again we can assume $x$ and $y$ are proper.

By \cref{build-rec},
\begin{multline}\label{3col-sum}
\sum_{a,b\in[3]} B(xaby)=\\
\sum_{x_m\neq a\neq b \neq y_1}\biggl[ \sum_{i=1}^m B(\widehat{x}_i aby)
+B(xby)+B(xay)+ \sum_{j=1}^n B(x ab\widehat{y}_j). \biggr]
\end{multline}
As in the previous proof, for $i\leq m-1$ the inductive hypothesis gives
$$\sum_{x_m\neq a\neq b \neq y_1} B(\widehat{x}_i aby)
= 2\binom{m+n+3}{m+1} B(\widehat{x}_i) B(y).
$$
The $i=m$ term must be combined with the next term, $B(xby)$, and we again
consider two cases.

Firstly, suppose $x_m=y_1=1$ (say).  Then
\begin{align*}
&\sum_{1\neq a\neq b \neq 1} B(\widehat{x}_m aby) +
\sum_{1\neq a\neq b \neq 1} B(xby)\\
=&\sum_{ab=23,32} B(\widehat{x}_m aby)
+\sum_{b=2,3} B(\widehat{x}_m 1by)\\
=&\sum_{a,b\in[3]} B(\widehat{x}_m aby)=2\binom{m+n+3}{m+1} B(\widehat{x}_m) B(y),
\end{align*}
by the inductive hypothesis.

Secondly, suppose $x_m=1\neq 2=y_1$ (say).  Then
\begin{align*}
&\sum_{1\neq a\neq b \neq 2} B(\widehat{x}_m aby) +
\sum_{1\neq a\neq b \neq 2} B(xby)\\
=&\sum_{ab=21,23,31} B(\widehat{x}_m aby)
+B(\widehat{x}_m 13y)\\
=&\sum_{a,b\in[3]} B(\widehat{x}_m aby)=2\binom{m+n+3}{m+1} B(\widehat{x}_m) B(y).
\end{align*}

The third and forth terms appearing on the right of \eqref{3col-sum} can be
treated symmetrically, so by \cref{build-rec} the entire sum becomes
$$2\biggl[\binom{m+n+3}{m+1}+\binom{m+n+3}{m+2}\biggr]B(x)B(y),$$
which equals the required expression.
\end{proof}

The following fact is not needed for our main results, but
it will imply that the $q$-color analogue of our processes
is not finitely dependent for $q\notin\{ 3,4\}$.

\begin{prop}\label{two-point}
Let $q\geq 2$ and $n\geq 0$.  We have
\begin{equation}\label{two-point-eq}
\sum_{x\in[q]^n}\bigl[B(1x2)-B(1x1)\bigr]=2\prod_{k=1}^n\bigl[k(q-2)-2\bigr].
\end{equation}
\end{prop}

\begin{proof}
We use $*$'s to denote unrestricted symbols, so
$B(a *^n b):=\sum_{x\in[q]^n} B(axb)$, etc.  Let $n\geq 1$.  By \cref{build-rec},
\begin{align*}
B(1 *^n 1)&= \sum_{\substack{x\in[q]^n:\\1x1 \text{ proper}}}\Bigl[
B(x1)+\sum_{i=1}^n B(1\hat x_i 1) +B(1x) \Bigr].
\end{align*}
But, by symmetry,
$$\sum_{\substack{x\in[q]^n:\\1x1 \text{ proper}}} B(x1)
=\sum_{a\neq 1} B(a *^{n-1} 1)=(q-1) B(1 *^{n-1} 2),$$ and the term $B(1x)$
can be treated similarly.  On the other hand,
$$\sum_{\substack{x\in[q]^n:\\1x1 \text{ proper}}} B(1\hat x_i 1)
=(q-2) B(1 *^{n-1} 1),$$ since each proper coloring of the form $1*^{n-1} 1$
arises from exactly $q-2$ proper colorings of the form $1 *^n 1$ by deleting
the $(i+1)$st symbol -- the two neighboring colors must be distinct, so there
are $q-2$ choices for the symbol between them that is deleted.

Therefore,
\begin{align*}
B(1 *^n 1)&=n(q-2)\,B(1 *^{n-1} 1) +2(q-1)\,B(1*^{n-1}2),
\intertext{and a simlar argument gives}
B(1 *^n 2)&= (n+2)(q-2)\,B(1*^{n-1}2) + 2\,B(1 *^{n-1} 1).
\end{align*}
Subtracting yields
$$B(1*^n2)-B(1*^n1)=
\bigl(n(q-2)-2\bigr)\bigl[B(1*^{n-1}2)-B(1*^{n-1}1)\bigr],$$
and induction finishes the proof.
\end{proof}

\section{The colorings}
\label{colorings}

\begin{proof}[Proof of \cref{main}]
Recall that $B(x)$ denotes the number of proper buildings of a word $x$. To
construct the $4$-coloring, we define
\begin{equation}\label{prob}
P(x)=P_4(x):=\frac{B(x)}{\Sigma(4,n)}
=\frac{B(x)}{(n+1)!\;2^n},\quad x\in[4]^n.
\end{equation}
We claim that there is a stationary $1$-dependent $4$-coloring $X$ with
cylinder probabilities given by
\begin{equation}\label{cylinder}
\P\bigl[(X_{i+1},\ldots,X_{i+n})=x\bigr]=P(x),\quad i,n\in\Z, \; x\in[4]^n.
\end{equation}

\cref{build-dep} gives that for all words $x$ and $y$,
\begin{equation}\label{dep-p}
\sum_{a\in[4]} P(xay) = P(x)P(y).
\end{equation}
Taking $y$ or $x$ to be the empty word $\emptyset$ gives respectively
$\sum_{a\in[4]} P(xa) = P(x)$ and $\sum_{a\in[4]} P(ay) = P(y)$, so
\eqref{cylinder} gives a consistent family of measures. We have
$P(\emptyset)=1$, and of course we have $P(x)\geq 0$ for all $x$. Thus by the
Kolmogorov extension theorem (see e.g.\ \cite[Theorem~6.16]{kallenberg})
there exists a process $X$ satisfying \eqref{cylinder}, and \eqref{cylinder}
immediately shows that it is stationary.  The process $X$ is a $4$-coloring
since $P(x)=0$ when $x$ is not a proper coloring, and \eqref{dep-p} gives
that it is $1$-dependent.

\sloppypar The construction of the stationary $2$-dependent $3$-coloring is
essentially identical.  We take
\begin{equation}\label{prob3}
P_3(x):=\frac{B(x)}{\Sigma(3,n)}
=\frac{2 B(x)}{(n+2)!},\quad x\in[3]^n.
\end{equation}
Consistency follows from \cref{build-cons}, and $2$-dependence from \cref{build-dep}.
\end{proof}

\begin{proof}[Proof of \cref{properties}(i,v)]
The symmetry and conditioning properties are immediate from
\eqref{prob},\eqref{prob3}, and the definition of proper
buildings.
\end{proof}

Via \cref{build-cons}, the above proof in fact shows that
for every $q\geq 2$ there is a symmetric, reversible,
stationary $q$-coloring $X$ given by
$$\P\bigl[(X_{i+1},\ldots,X_{i+n})=x\bigr]=\frac{B(x)}{\Sigma(q,n)}.$$
It is immediate that this matches the description of the process via
conditioning given in the introduction.  The event $E$ that the random
permutation $\sigma$ is a proper building of the random word $Z$ has
probability $\Sigma(q,n)/(n!q^n)$, which is $(n+1)/2^{n}$ for $q=4$ and
$\binom {n+2}2 /3^n$ for $q=3$.

 Here is an alternative description
of this process that does not involve conditioning, and
that provides a practical and efficient method for exact
sampling.  Start with a sequence of length $1$ consisting
of a uniformly random element of $[q]$.  At each step,
insert a new color, in such a way that the sequence is
always a proper coloring, as follows.  Given that the
current sequence has length $n-1$, choose one of the $n-2$
locations between two consecutive elements each with
probability \mbox{$(q-2)/[n(q-2)+2]$}, or one of the $2$
end locations each with probability $(q-1)/[n(q-2)+2]$.
Then insert a color in the chosen location, chosen
uniformly from among those that will still result in a
proper coloring; there are $q-2$ choices at an internal
location, or $q-1$ at an end.  It is easily seen that the
resulting sequence after $n-1$ such steps has the same law
as $(X_1,\ldots,X_n)$. See \cite{mallows-shepp,nakata} for
a somewhat related process.

\sloppypar \cref{two-point} shows that for $q\notin\{3,4\}$ the process is
{\em not} $k$-dependent for any $k$.  Indeed, the right side of
\eqref{two-point-eq} is positive for all $q\geq 5$ and $n\geq 0$ (the product
over $k$ begins $(q-4)(2q-6)(3q-8)\cdots$), so the events $X_i=1$ and $X_j=1$
are strictly negatively correlated for $i\neq j$ when $q\geq 5$.  (The case
$q=2$ is trivial).

\section{Block-factors}
\label{block}

\begin{proof}[Proof of \cref{block-factor}]
Let $U_1,\ldots,U_{r+1}$ be i.i.d.\ random variables, and let $f:\R^r\to [q]$
be a measurable function.  We claim that for all $r,q\geq 1$,
\begin{equation}\label{pos-prob}
\P\bigl[f(U_1,\ldots,U_r)=f(U_2,\ldots,U_{r+1})\bigr]>0.
\end{equation}
Once this is proved, the required result follows immediately.

We prove \eqref{pos-prob} by induction on $r$.  For $r=1$ it is immediate,
since $f(U_1)$ and $f(U_2)$ are i.i.d.  Assume that it holds for $r-1$ and
all $q$.  Now for $f:\R^r\to [q]$ define
$$S(u_1,\ldots,u_{r-1}):=
\Bigl\{ a\in [q]: \P\bigl[f(u_1,\dots,u_{r-1},U_r)=a\bigr]>0\Bigr\},$$ i.e.,
the set of values that $f$ can take with positive probability given its first
$r-1$ arguments.  Since the function $S$  takes at most $2^q$ values, the
inductive hypothesis gives
$$\P\bigl[S(U_1,\ldots,U_{r-1})=S(U_2,\ldots,U_r)\bigr]>0.$$
Moreover, since a.s.\ $f(U_1,\ldots,U_r)\in S(U_1,\ldots,U_{r-1})$, we can
find deterministic $A\subseteq [q]$ and $a\in A$ such that
$$\P\bigl[S(U_1,\ldots,U_{r-1})=S(U_2,\ldots,U_r)=A,\; f(U_1,\ldots,U_r)=a\bigr]>0.$$
Using the definition of $S(U_2,\ldots,U_r)$, and the fact that $U_{r+1}$ is
independent of $(U_1,\ldots,U_r)$, the conditional probability that
$f(U_2,\ldots,U_{r+1})=a$ given the above event is positive. Thus,
\[
\P\bigl[f(U_1,\ldots,U_{r})=f(U_2,\ldots,U_{r+1})\bigr]>0.\qedhere
\]
\end{proof}

By replacing ``$>0$'' with ``$>\epsilon$'' in the
definition of $S$, the above proof can be made
quantitative, giving that the left side of \eqref{pos-prob}
is at least
$$\frac 1{2^{2\rule{0pt}{7pt}^{\iddots
\rule{0pt}{9pt}^{2^{4q}}}}},$$ where there are $r-1$
exponentiation operations in the tower.  The tower-function
form of this bound is sharp.  See \cite{hsw} for more
information.


\section{Hilbert spaces and hidden-Markov processes}
\label{hilbert}

In this section we present the Hilbert space connection that leads to
\cref{analysis}, and from which we will also deduce \cref{hidden} concerning
hidden-Markov processes.

Before doing this we give the much simpler proof of a special case of
\cref{hidden}: a stationary $k$-dependent $q$-coloring cannot itself be a
Markov chain. Indeed, let $\mathbf{P}=(P_{a,b})_{a,b\in[q]}$ be its
transition matrix.  Since $X_n$ is independent of $X_0$ for $n>k$, the
conditional law of $X_n$ given $X_0$ is simply the stationary distribution of
the Markov chain, so in particular the conditional laws of $X_{k+1}$ and
$X_{k+2}$ given $X_0$ are identical, hence $P^{k+1}=P^{k+2}$, i.e.\
$P^{k+1}(1-P)=0$. Therefore the eigenvalues of $P$ are precisely $0$ and $1$.
However, since $X$ is a proper coloring we have $P_{a,a}=0$ for all $a$, so
$P$ has trace $0$, and its eigenvalues (with multiplicities) sum to $0$, a
contradiction.

The proof of \cref{hidden} follows a broadly similar strategy, but requires a
more elaborate set-up, which also gives \cref{analysis}.  Let $X=(X_i)_{i\in
\Z}$ be a stationary process taking values in $\Omega:=[q]^\Z$, with law
$\mu$.  Let $L^2$ be the Hilbert space of real $L^2(\mu)$ functions on
$\Omega$ (which is separable by the Stone-Weierstrass and Lusin theorems).
Let $S:\Omega\to\Omega$ be the shift map given by $S(x)_j=x_{j-1}$, and
define the shift operator $T:L^2\to L^2$ by $(Tf)(x)=f(S^{-1}(x))$. Let $A$
be the space of functions $f\in L^2$ that depend only on $x_0,x_1,\ldots$,
and let $B$ be the space of functions $f\in L^2$ that depend only on
$\ldots,x_{-1},x_0$. Thus $T A\subseteq A$ and $TB \supseteq B$.  Let $P_B$
denote orthogonal projection in $L^2$ onto $B$, or in probabilistic terms,
$P_B(f)=\E (f\mid \ldots ,X_{-1},X_0)$. Define
$$U:= \overline{P_B A}$$
(where the bar denotes closure), and define $R$ to be the
restriction
$$R:= (P_B T)|_U.$$

\begin{lemma}\label{reduction}
Let $X=(X_i)_{i\in \Z}$ be a stationary process taking values in
$[q]^\Z$. Define the Hilbert space $U$ and the operator $R$ as above.
\begin{ilist}
\item We have $RU\subseteq U$.
\item If $X$ is $k$-dependent, then $R^n U$ is the space of constant
    functions, for all $n>k$.
\item If $X$ is a $q$-coloring, then $U$ has an orthogonal decomposition
$$U=U_1\oplus \cdots \oplus U_q$$
into closed linear subspaces such that $RU_j$ is orthogonal to $U_j$ for
each $j$.
\end{ilist}
\end{lemma}

\begin{proof}
We claim first that
\begin{equation}\label{conjugate}
P_B T P_B = P_B T.
\end{equation}
Indeed, let $f\in L^2$ and $g=P_B f$. Then $g-f$ is orthogonal to $B$. Since
$T$ is an isometry, $T(g-f)$ is orthogonal to $TB$. Since $TB\supseteq B$, in
particular $T(g-f)$ is orthogonal to $B$. Thus, $P_B T (g-f)=0$. This gives
\eqref{conjugate}.

Now suppose that $f\in A$ and $g=P_B f$.  Then
\eqref{conjugate} gives $R g=R P_B f=P_B T P_B f=P_B T f
\in P_B A$.  Thus $R$ maps $P_BA$ into itself.  Since $R$
is continuous, the same applies to the closure $U$,
establishing~(i).

A similar argument to the above gives $R^n U\subseteq
\overline{P_B T^n A}$ for every integer $n\geq 1$.  Now if
$X$ is $k$-dependent then $P_B T^n A$ is the space of
constants for all $n>k$, so we obtain (ii).

Finally, let $V_j$ denote the space of functions in $L^2$
that are supported on the set of $x\in \Omega$ such that
$x_0=j$.  Let
$$U_j:= \overline{P_B(V_j\cap A)}.$$
Then $U_j\subseteq V_j$, since $P_B V_j\subseteq V_j$ and
$V_j$ is closed.  The spaces $V_j$ are mutually orthogonal,
therefore so are $U_j$.  Clearly, $A$ is the direct sum of
the subspaces $V_j\cap A$, and therefore $P_B A$ is spanned
by the spaces $P_B(V_j\cap A)$.  Since these are mutually
orthogonal, the same applies to the closures.  So $U$ is
the orthogonal direct sum of the spaces $U_j$.

Now suppose that $X$ is a $q$-coloring; then $V_j$ is
orthogonal to $TV_j$.  To prove (iii) we must show that
$RU_j$ and $U_j$ are orthogonal.  Suppose $f,g\in U_j$.
Then $\langle f, Rg\rangle= \langle f,P_B T g\rangle =
\langle P_B f,T g\rangle = \langle f,T g\rangle = 0$. (Here
we used that $P_B$ is an orthogonal projection and
therefore self-adjoint, and that $f,g\in V_j$ so $f$ and
$Tg$ are orthogonal).  This proves~(iii).
\end{proof}

\begin{proof}[Proof of \cref{analysis}]
This is immediate by \cref{main,reduction}.
\end{proof}

To prove \cref{hidden} we also need the following.

\begin{lemma}\label{finite-dim}
If $X$ is a hidden-Markov process then the Hilbert space $U$ defined above
has finite dimension.
\end{lemma}

\begin{proof}
Let $X$ be a function of a Markov chain $M$ with state
space $S$.  Consider the earlier space $L^2=L^2(\mu)$
embedded in the possibly larger space of $L^2(\lambda)$
functions on the probability space of $M$, where $\lambda$
is the law of $M$, and where we now interpret a function
$f\in L^2(\mu)$ as the random variable $f(X)$. Let $C$ be
the space of random variables in $L^2(\lambda)$ that depend
only on $\ldots,M_{-1},M_0$, and let $P_C$ denote
orthogonal projection onto $C$. Since $X_i$ is a function
of $M_i$ we have $B\subseteq C$, and therefore
$U=\overline{P_B A} =\overline{P_B P_C A}$, so it suffices
to prove that $P_C A$ is finite-dimensional.  Let $f\in A$.
Then
$$P_C f = \E(f\mid \ldots,M_{-1},M_0)=\E (f\mid M_0),$$
by the Markov property.  But the latter depends only on $M_0$, so it is in
the linear span of the functions $\{\ind[M_0 = s]:s\in S\}$.  Thus $\dim (P_C
A) \leq |S|$.
\end{proof}

\begin{proof}[Proof of \cref{hidden}]
\sloppypar Apply \cref{reduction,finite-dim}.  Since $U$ is
finite-dimensional, choose an orthonormal basis $e_1,\ldots,e_d$ that
comprises orthonormal bases for each $U_j$.  Since $R e_i$ is orthogonal to
$e_i$ for each $i$ we have $\trace(R)=0$.  But \cref{reduction}(ii) implies
that $R$ has exactly one non-zero eigenvalue, a contradiction.
\end{proof}

Hilbert space representations of $k$-dependent processes were also explored
in \cite{devalk}.  We briefly discuss the connection with the above approach.
It is shown in \cite{devalk} that if $X$ is a stationary $k$-dependent
$[q]$-valued stochastic process, there exist a Hilbert subspace $H$ of $L^2$
and bounded linear operators $A_1, \dots, A_q$ on $H$ that encapsulate the
cylinder probabilities of $X$ via
$$P((X_1,\dots, X_n)=x)=\langle A_{x_1}\cdots A_{x_n}\mathbf{1},
\mathbf{1}\rangle$$
with the subsidiary conditions
\begin{gather}
(A_1+\cdots +A_q)^kh=\langle h,\mathbf{1}\rangle\mathbf{1},
\quad h\in H,\label{cond1}\\
(A_1+\cdots +A_q)\mathbf{1}=\mathbf{1},\nonumber\\
(A_1^*+\cdots +A_q^*)\mathbf{1}=\mathbf{1}\nonumber,
\end{gather}
where $\mathbf 1$ is the function that is identically 1. The subspace $H$ is
not given explicitly in \cite{devalk}, though the operators $A_i$ are. The
construction above provides an explicit choice:
$$H=\overline{RU},\quad A_i=P_HI_iT,$$
where $I_i=\ind[X_1=i]$.  (These $A_i$'s are the same as in \cite{devalk}.)
To check \eqref{cond1}, for example, take $h\in H$ and note that, since
$H\subseteq B$, we have $P_BTh=Rh\in R^2U\subseteq H$, so that $P_HTh=Rh$.
Iterating gives $(P_HT)^nh=R^nh$ for $n\geq 1$. Since $A_1+\cdots+A_q=P_HT$,
\cref{reduction}(ii) gives \eqref{cond1}.

\section{One-color marginals}
\label{renewal-structure}

\cref{properties}(ii) is a consequence of the following more general result
that in any $1$-dependent coloring, the set of locations of a single color
has a simple structure.

\begin{prop}\label{renewal}\sloppypar
Suppose that $(X_i)_{i\in\Z}$ is a stationary $1$-dependent $q$-coloring.
Suppose $p:=\P(X_0=1)>0$. Then the process $J$ defined by $J_i:=\ind[X_i=1]$
is a renewal process, and its renewal time $T$ (the number of steps between
consecutive $1$'s) has probability generating function
$$G(s):=\E s^T=\frac{p s^2}{1-s+p s^2}.$$
\end{prop}

The fact that $J$ is a renewal process is due to Fuxi Zhang.  We are grateful
for her permission to include it.

\begin{proof}[Proof of \cref{renewal}]
To prove that $J$ is a renewal process we must check that $(J_i)_{i<0}$ and
$(J_i)_{i>0}$ are conditionally independent given $J_0=1$. Since $X$ is a
coloring, $J_0=1$ implies $J_{-1}=J_1=0$.  For a string $u\in\{0,1,*\}^n$ we
write $\P(u):=\P(J_i=u_i \; \forall i\text{ s.t. } u_i\neq *)$ (so that $*$'s
denote unrestricted symbols).  Let $u,v\in\{0,1\}^{n-1}$ be any binary words.
Then
\begin{align*}
\P(u010v)&=\P(u{*}1{*}v)\\
&=p\;\P(u)\;\P(v)\\
&=p^{-1}\;\P(u{*}1)\;\P(1{*}v)\\
&= p^{-1}\;\P(u01)\;\P(10v)
\end{align*}
(where in the 2nd and 3rd equalities we used $1$-dependence of $J$, and in
the 1st and 4th we used the fact that $J$ has no consecutive $1$'s).  Now
dividing through by $p$ shows that the events $(J_{-n},\ldots,J_{-1})=u0$ and
$(J_1,\ldots,J_n)=0v$ are conditionally independent given $J_0=1$, as
required.

Turning to the renewal time distribution, we write
$$p_n=\P(1 0^{n-1} 1)/p.$$
This is the conditional probability given that we have just seen $1$ of
waiting $n$ steps until the next $1$, thus $(p_n)_{n\geq 1}$ is the
probability mass function of the renewal time.  Note that $p_1=0$. The
probability generating function is defined by
$$G(s):=\sum_{n\geq 1} p_n s^n.$$

Since $J$ is a renewal process, for any integers $k_i>0$ we have
\begin{equation}
\P(1 0^{k_1-1} 1 0^{k_2-1} 1 \cdots 0^{k_m-1} 1) = p \; p_{k_1} p_{k_2}
\cdots p_{k_m}. \label{renewal-product}
\end{equation}
We claim that
\begin{equation}
p\bigl(G(s) + G(s)^2 + G(s)^3 + \cdots\bigr) = p^2\bigl(s^2 + s^3
+s^4 +\cdots\bigr). \label{g-recurrence}
\end{equation}
To check this, observe that by \eqref{renewal-product}, the coefficient of
$s^n$ on the left side is the sum of $\P(1 u 1)$ over all binary strings $u$
of length $n-1$.  But this is simply $\P(1 *^{n-1} 1)$, which equals $0$ for
$n=1$ (by the coloring property) and $p^2$ for $n\geq 2$ (by $1$-dependence),
as required for the right side.

Finally, summing the geometric series in \eqref{g-recurrence} and solving
gives the claimed formula for $G(s)$.
\end{proof}

\cref{renewal} yields an alternative proof of the following result of Schramm
(see \cite{hsw} for Schramm's original proof).

\begin{cor}\label{quarter}
In any stationary $1$-dependent $q$-coloring, any given color has marginal
probability at most $1/4$.  In particular there is no stationary
$1$-dependent $3$-coloring.
\end{cor}

\begin{proof}
Suppose that $p>1/4$.  Then both singularities of $G$ (viewed as a function
on the complex plane) are complex.  This contradicts a theorem of Pringsheim
from 1893 (see \cite[Theorem~IV.6]{flajolet-sedgewick} or
\cite[\S~7.21]{titchmarsh}): a Taylor series with non-negative real
coefficients and finite radius of convergence $R$ has a singularity at $R$.
\end{proof}

We remark that the possibility of a stationary $1$-dependent $3$-coloring can
also be ruled out without appeal to Pringsheim's theorem as follows.  In the
Taylor series for $G$, the coefficient of $s^7$ is $p(1-p)(1-3p)$, which
forces $p\leq 1/3$. But if $p=1/3$ then the coefficient of $s^8$ is
$-1/81<0$.

\begin{proof}[Proof of \cref{properties}(ii)]
We prove that any stationary $1$-dependent $4$-color\-ing has property (ii),
as claimed at the end of \cref{properties}. By \cref{quarter}, each color
must have marginal probability exactly $p=1/4$, in which case the probability
generating function of the renewal time in \cref{renewal} factorizes to
become
$$G(s)=\Bigl(\frac{s}{2-s}\Bigr)^2.$$
But this is the probability generating function of the sum of two independent
$\mbox{Geometric}(1/2)$ random variables, which yields the claimed
description of the process $J$.
\end{proof}

One straightforward consequence of \cref{properties}(ii) is that for any
stationary $1$-dependent $4$-coloring $X$,
$$\P\bigl(X_1,\ldots,X_n \in\{2,3,4\}\bigr)=\frac{n+2}{2^{n+1}}.$$
For our $4$-coloring this also follows from \cref{build-tot} with $q=3$ (and
symmetry).

\cref{quarter} and its proof reflect the fact that $q=4$
colors is in a sense a critical case for the $1$-dependent
coloring problem.  This is one reason for our belief that
the solution is unique.  See \cref{sec-hardcore} for
extensions of some of these ideas to general graphs.

Finally in this section we derive the claimed description
of the one-color marginal for the $3$-coloring, for which
we need to return to proper buildings.
\begin{proof}[Proof of \cref{properties}(iv)]
It suffices to check that the two processes have equal
probabilities of assigning $1$'s to every integer in a
finite set $A\subset\Z$, since all other cylinder
probabilities can be computed from these by
inclusion-exclusion. Since both processes are $2$-dependent
and have no adjacent $1$'s, it is enough to do this for $A$
of the form $\{1,3,\ldots,2m-1\}$.

Let $P(x)=P_3(x)=2B(x)/(n+2)!$ denote the cylinder
probability of the $3$-coloring for the word $x\in[3]^n$.
We use $*$'s to denote unrestricted symbols in $[3]$ to be
summed over, so that $2$-dependence of the process says
that $P(x{*}{*}y)=P(x)P(y)$ for all words $x$ and $y$.
\cref{build-rec} gives that for every proper coloring
$x\in[3]^n$,
\begin{equation}\label{recur-3col}
(n+2) P(x) = \sum_{i=1}^n P(\widehat x_i).
\end{equation}
Write $p_m:=P(1{*}1{*}1\cdots{*}1)$, where the word has $m$
$1$'s and length $2m-1$, and $p_0:=1$.  Then,
\begin{align*}
(2m+1)p_m&=P({*}1{*}1{*}1\cdots)+P(1{*}{*}1{*}1\cdots)
+P(1{*}1{*}{*}1\cdots)+\cdots\\
&=p_0 p_{m-1} +p_1 p_{m-2} +\cdots + p_{m-1} p_0.
\end{align*}
(The first equality requires some care: the left side does
not change if we interpret each $*$ as being summed over
$\{2,3\}$ instead of $[3]$; then we can apply
\eqref{recur-3col}.  The words that arise from deleting a
$*$ vanish, since they are not proper colorings, and in the
others we may allow each $*$ to revert to its original
meaning, since it is still adjacent to a $1$.  For the
second equality we use $2$-dependence).

We now show that the cylinder probabilities of the second
process satisfy the same recurrence, whereupon induction
will finish the proof.  Indeed, let
$q_m:=\P(U_1<U_2>U_3<\cdots >U_{2m+1})$, where the
inequalities alternate, and $q_0:=1$.  This equals the
probability of the event $E$ that the elements of a
uniformly random permutation $\pi$ in $S_{2m+1}$ satisfy
the same inequalities.  We decompose $E$ according to the
location of the maximum of $\pi$. The conditional
probability of $E$ given $\pi_{2i}=2m+1$ is
\[\P(\cdots
<\pi_{2i-2}>\pi_{2i-1})\,\P(\pi_{2i+1}<\pi_{2i+2}>\cdots )=
 q_{i-1}q_{m-i}.\qedhere\]
\end{proof}

\section{Alternative formula}
\label{alternative}

In this section we derive a different formula for the cylinder probabilities
of the $1$-dependent $4$-coloring $X$ of $\Z$.  It was this formula that
originally convinced us that such a coloring must exist (contrary to much
circumstantial evidence), since it has all the required properties, except
that it appears extremely difficult to prove directly that it is nonnegative.
We were led to our solution by seeking recursions satisfied by this formula,
and finding the equivalent of Lemma~\ref{build-rec} (which we then
re-interpreted via buildings).  Below we state the formula, after some
necessary definitions. We then discuss applications and motivation before
giving the proof.  The basic idea is to start with a postulated law for the
$1$-dependent binary process $(\ind[X_i=1\text{ or }2])_{i\in\Z}$, and try to
build the law of $X$ around it.

We identify the $4$ colors with binary strings of length $2$.  It is
convenient to use the binary symbols $+(=+1)$ and $-(=-1)$, and to write the
strings as column vectors, so
$1,2,3,4=\binom\m\m,\binom\m\p,\binom\p\m,\binom\p\p$ (say; the choice of
bijection is immaterial). Then a word $x\in [4]^n$ becomes a $2\times n$
matrix, and we denote its rows $y,z\in\{\m,\p\}^n$:
$$x=(x_1,x_2,\ldots, x_n) = \binom yz=
\left(\begin{matrix} y_1&y_2&\cdots&y_n \\ z_1&z_2&\cdots&z_n\end{matrix}\right).
$$

Let $y\in\{\m,\p\}^n$, and let $\alpha(y)$ denote the number of permutations
$\pi\in S_{n+1}$ such that $\pi_i<\pi_{i+1}$ if $y_i=+$, and
$\pi_i>\pi_{i+1}$ if $y_i=-$, for each $1\leq i\leq n$ (in other words, the
number of permutations with descent set given by the locations of $\m$'s, or
the number of linear extensions of the $(n+1)$-element poset generated by
these inequalities). For example,
$$
\arraycolsep=1pt
  \begin{array}{lrlcccccccccccccccl}
   \text{if}\quad   &y&=&\p&&\m&&\p&&\p&& \\
   \text{then}\quad &\alpha(y)&=\bigl|\bigl\{\pi\in S_{5}:\;
   \pi_1 &<&\pi_2 &>& \pi_3 &<& \pi_4 &<& \pi_5 &\bigr\}\bigr| = 9.
  \end{array}
$$
(See e.g.\ \cite{niven} for information about $\alpha$). If $(U_i)_{i\in\Z}$
are i.i.d.\ Uniform on $[0,1]$ and we let $Y_i:=(-1)^{\ind[U_i>U_{i+1}]}$
then $\P((Y_1,\ldots,Y_n)=y)=\alpha(y)/(n+1)!$.  This will be the law of $Y$,
where $X=\binom YZ$.

A \df{Dyck word} of length $2k$ is an element of $\{-,+\}^{2k}$ comprising
$k$ $+$'s and $k$ $-$'s, such that the $i$th $+$ precedes the $i$th $-$ for
each $i$. A \df{dispersed Dyck word} of length $m$ is an element of
$\{-,0,+\}^m$ that is a concatenation of Dyck words and strings of $0$'s.
Examples of dispersed Dyck words are $\p\m0\p\p\m\m00$, $000$, and $\p\m\p\m$
(but not $\p 0\m$). Let $\DD(m)$ be the set of dispersed Dyck words of length
$m$, and for $w\in \DD(m)$, let $|w|$ be the number of $+$'s in
$w$.%
\footnote{We remark that $|\DD(m)|=\binom{m}{\lfloor m/2\rfloor}$, although
we will not use this. For a bijective proof, consider a lattice path from
$(0,\tfrac12)$ to $(m,\pm\tfrac12)$ via steps $(1,\pm1)$.  Map steps between
heights $-\tfrac12$ and $\tfrac12$ to $0$'s, and reflect excursions below
$-\tfrac12$ into excursions above $\tfrac12$.}

If $y\in\{-,+\}^n$ has $m$ intervals of constancy (or \df{runs}) and $w\in
\DD(m-1)$, define $y_w\in \{-,+\}^n$ to be the word obtained by changing the
signs of some whole runs of $y$, not including the first and last runs, in
such a way that the $j$th sign-change between runs is eliminated precisely
for those $j$ with $w_j\neq 0$.  For example, with $n=15$ and $m=9$,
$$
\arraycolsep=1pt
  \begin{array}{lrccccccccccccccccl}
   \text{if}\quad &w=\;&&\p&&\p&&\m&&\m&&0&&\p&&\m&&0&\\
   \text{and}\quad &y=\;&\p\p\p&&\m\m&&\p&&\m&&\p\p&&\m\m\m&&\p&&\m&&\p \\
    \text{then}\quad &y_w=\;&\p\p\p&&\p\p&&\p&&\p&&\p\p&&\m\m\m&&\m&&\m&&\p ,
  \end{array}
$$
(where the horizontal spacing emphasizes the runs of $y$). Note that $y_w$
depends on $w$ only through the locations of its Dyck words, not on which
words they are, so for instance $y_{\p\p\m\m0\p\m0}=y_{\p\m\p\m0\p\m0}$.

Now let $y,z\in\{-,+\}^n$, and let $m$ be the number of runs of $y$. For
$1\leq j\leq m-1$, let $\ell_j$ and $r_j$ be respectively the elements of $z$
immediately before and after the $j$th sign-change in $y$. For example, if
$$\binom yz =
\arraycolsep=1pt
  \left(\begin{array}{cccccccccccccl}
    \p&\p&\ \ &\m&\m&\m&\ \ &\p&\p&\ \ &\m&\ \ &\p&\p \\
    z_1&z_2&\ &z_3&z_4&z_5&\ &z_6&z_7&\ &z_8&\ &z_9&z_{10}
  \end{array}\right)
$$
then $\ell_1=z_2$, $r_1=z_3$, and $r_3=\ell_4=z_8$, etc. Let
$$c(w,y,z):=\prod_{j=1}^{m-1}\begin{cases}\ell_j,&
w_j=+;\\r_j,&w_j=-;\\1,&
w_j=0.\end{cases}$$

We are now ready to state the formula.  For $x=\binom yz\in[4]^n$, where $y$
has $m$ runs, define
\begin{multline}\label{formula}
Q(x)= Q\binom yz:= \\
\begin{cases} \displaystyle 2^{n-m}\!\!\!\!\!\sum_{w\in
\DD(m-1)}\!(-1)^{|w|}c(w,y,z)\, \alpha(y_w) & \text{ \parbox{1.15in}{if $x$
is a
\\ proper coloring;}}\\0 &\text{ otherwise}.
\end{cases}\end{multline}

\begin{thm}\label{alt}
For $x\in [4]^n$ we have $B(x)=Q(x)$.
\end{thm}
In consequence, the cylinder probabilities $P(x)$ for the $4$-coloring $X$ of
Theorem~\ref{main} can of course be expressed as $P(x)=Q(x)/[2^n(n+1)!]$.
Theorem~\ref{alt} will be proved by showing that $Q(x)$ satisfies the same
recurrence as $B(x)$ (\cref{build-rec}).  It is now easy to deduce the
claimed marginal distribution for the first binary digit.
\begin{proof}[Proof of \cref{properties}(iii)]
We claim that
\begin{equation}\label{marginal}
\sum_{z\in \{\m,\p\}^n} Q\binom yz = 2^{n}\alpha(y), \qquad y\in \{\m,\p\}^n;
\end{equation}
then the result is immediate from Theorem~\ref{alt}.

To prove \eqref{marginal}, sum \eqref{formula} over $z$ and interchange the
order of summation.  The contribution from the trivial word $w=00\cdots0$ is
$$\sum_{z:x\text{ is proper}}2^{n-m}\alpha(y)=2^n\alpha(y),$$
since $z$ must alternate within each run of $y$, and thus there are $2^m$
choices. The contribution from every other $w$ vanishes. To see this, fix a
nontrivial $w$, and consider the location of the first $+$ in $w$. For any
$z$, let $z'$ be obtained from $z$ by changing the sign of every symbol in
the run of $y$ that precedes that $+$.  Then $c(w,y,z')=-c(w,y,z)$, so the
terms corresponding to $z$ and $z'$ cancel.
\end{proof}

Theorem~\ref{alt} implies a host of combinatorial identities; we briefly
highlight some examples.  Re-interpreting the result proved above in terms of
buildings gives the following.  For $y\in \{\m,\p\}^n$, define
$S(y)\subset[4]^n$ to be the Cartesian product
$$S(y):=\bigtimes_{i=1}^n\begin{cases}
\{1,2\},& y_i=-;\\
\{3,4\},& y_i=+.
\end{cases}$$
Then we have
$$\sum_{x\in S(y)} B(x) = 2^n \alpha(y),\qquad y\in\{\m,\p\}^n.$$
When $y=\p\p\cdots\p$ this is \cref{build-tot} with $q=2$,
but it seems much less clear why the general case holds.
Can it be given a bijective proof?  Taking $y$ alternating
of even length and combining with \cref{properties}(iv)
yields the curious identity
$$\sum_{x\in(\{1,2\}\times\{3,4\})^n} B(x) =
\frac{4^n}{n+1} \sum_{x\in(\{1,2\}\times\{3\})^n} B(x), \qquad n\geq 1.$$

The $S_4$-symmetry of $B(x)$ implies in particular that
$$Q\binom yz=Q\binom zy, \qquad y,z\in\{\m,\p\}^n.$$
Again, it does not seem at all clear how to prove this directly from the
definition \eqref{formula}.  For instance, in the very simplest case where
$z$ is a constant word and $y$ is alternating, it reduces to
$$
\sum_{\substack{m\geq 1,\; t_1,\ldots,t_m\geq 0:\\\sum_j(2t_j+1)=n}}
\Bigl[\prod_{j=1}^{m} (-C_{t_j})\Bigr] \, \alpha\bigl(2t_1+1,\ldots,2t_m+1\bigr) = 2^{n-1},
\quad n\geq 1,
$$
where $\alpha(k_1,\ldots,k_m)$ denotes $\alpha(y)$ for a word $y$ constructed
so as to have successive run lengths $k_1,\ldots,k_m$, and
$C_t:=\binom{2t}{t}/(t+1)$ are the Catalan numbers.  We have found a direct
proof of this last identity, but even this involves a fairly intricate
inclusion-exclusion argument for posets.

Another application of the formula \eqref{formula} is that it gives rise to a
computationally efficient method for computing the cylinder probabilities of
the $4$-coloring.  Indeed, there is a recurrence based on \eqref{formula}
that allows $Q(x)(=B(x))$ to be computed in $O(n^3)$ operations for a word
$x$ of length $n$, whereas a na\"ive application of \eqref{formula} requires
exponential time, as does computing $B(x)$ via \cref{build-rec}.  We state
this recurrence at the end of this section.

Before giving the proof of Theorem~\ref{alt} we briefly discuss how we
arrived at the formula \eqref{formula} (before knowing whether any
$k$-dependent $q$-coloring existed).  Suppose $X$ is a $1$-dependent
$4$-coloring, and decompose it into two binary sequences $X=\binom YZ$.  Then
$Y$ is a stationary $1$-dependent binary process.  The law of such a process
is determined by the sequence $v_n=P(Y_1=\cdots =Y_n=+)$, since all other
cylinder probabilities can be computed from $v$ by inclusion-exclusion.  Of
course, the sequence $v$ must satisfy certain inequalities in order that
these cylinder probabilities be nonnegative.  Many choices for $v$ are
possible. Examples are those for which $1,1,v_1,v_2,v_3,\dots$ is a P\'olya
frequency sequence -- see \cite[Chapter~8]{Karlin}.

\sloppypar Suppose for the purposes of the current discussion that $Y$ is any
stationary $1$-dependent binary process, and let $\alpha'$ be defined by
$\P[(Y_1,\ldots ,Y_n)=y]=\alpha'(y)/(n+1)!$.  By considering the constraints
imposed on the cylinder probabilities of $X$ by $1$-dependence, one is led
(after a certain amount of computation and some inspired guesses) to the
hypothesis that $\P[(X_1, \ldots, X_n)=x]=Q'(x)/[(n+1)!2^n]$, where $Q'$ is
given in terms of $\alpha'$ by the formula \eqref{formula}.  It is not
difficult to check that a $Q'$ defined in this way satisfies the equalities
required for consistency and $1$-dependence of $X$, for {\em any} $\alpha'$
arising from a stationary $1$-dependent $Y$.

The only issue is nonnegativity of $Q'(x)$.  This does not hold for general
$\alpha'$: for instance if $Y$ is i.i.d\ with $\P(Y_0=\p)=1/2$ then one can
check that $Q'(x)<0$ for $y=\p\m\p\m$ and $z=\p\p\p\p$.  In fact it appears
likely that $\alpha'=\alpha$ is the only choice that works.  However, it
seems extremely difficult to prove nonnegativity of $Q$ directly from
\eqref{formula} in that case.  The only way we know is to prove that $Q$
satisfies the same recurrence as $B$.

We now turn to the proof of \cref{alt}.  A key ingredient is that $\alpha$
satisfies a recurrence similar to the one that we wish to check for $Q$.  As
before, let $\alpha(k_1,\ldots,k_m)$ denote $\alpha(y)$ where $y$ is a binary
word with $m$ runs of successive lengths $k_1,\ldots,k_m$.  If one $k_i$ is
$0$ the interpretation is that the two neighboring intervals coalesce, so
that for example $\alpha(k_1,k_2,0,k_4,k_5)=\alpha(k_1,k_2+k_4,k_5)$ and
$\alpha(0,k_2,k_3,\ldots)=\alpha(k_2,k_3,\ldots)$.
\begin{prop}\label{beta-rec}
For positive integers $k_1,\dots,k_m$,
\begin{multline*} \alpha(k_1,k_2,\dots, k_m)\\
= \alpha(k_1-1,k_2,\dots,k_m)+\alpha(k_1,k_2-1,\dots,k_m)+
\cdots+\alpha(k_1,\dots,k_m-1).
\end{multline*}
\end{prop}
This is a special case of the main result of \cite{EHS},
when applied to the poset that defines $\alpha$. We also
give a simple direct proof.

\begin{proof}[Proof of \cref{beta-rec}]\sloppypar
Suppose $\alpha(k_1,\dots, k_m)=\alpha(y)$ where $y\in\{\m,\p\}^n$ is of
length $n=\sum_j k_j$.  Let $E$ be the set of permutations $\pi\in S_{n+1}$
that satisfy the inequalities in the definition of $\alpha(y)$, so
$\alpha(y)=|E|$. For $1\leq i\leq n+1$, let $E_i$ be the set of permutations
$\pi\in E$ that have their maximum at $i$, i.e.\ $\pi_i=n+1$.  For $1<i<n+1$
we further distinguish according to the order of the neighboring elements:
let $E_i^+$ be the set of $\pi\in E_i$ such that $\pi_{i-1}<\pi_{i+1}$, and
define $E_i^-$ similarly with the inequality reversed. Clearly,
$$E=E_1 \cup E_{n+1} \cup \bigcup_{1<i<n+1} (E_i^+\cup E_i^-),$$
and the union is disjoint.  However, $E_i$ is empty unless $\pi_i$ is already
a local maximum in the sequence of inequalities defining $E$ (i.e.\
$(y_{i-1},y_{i})=(\p,\m)$, where restrictions on ``$y_0$" and ``$y_{n+1}$"
are ignored).  In that case, we have
\[\arraycolsep=1pt\def\arraystretch{1.2}\begin{array}{rlccl}
|E^+_i|=&\alpha(k_1,k_2,\ldots, &k_{j-1},&k_{j}{-}1,&\ldots,k_m);\\
|E^-_i|=&\alpha(k_1,k_2,\ldots, &k_{j-1}{-}1,&k_j,&\ldots,k_m),
\end{array}\]
 when $1<i<n+1$ and
$(y_{i-1},y_i)=(\p,\m)$ is the boundary between the $(j-1)$st and $j$th runs
of $y$, and similar statements hold for $E_1$ and $E_{n+1}$.  (Indeed, the
maximum element $n+1$ in the permutation can be ignored, and the remaining
elements $1,\ldots,n$ satisfy precisely the inequalities required for the
appropriate ``reduced'' $\alpha$).
\end{proof}

\begin{proof}[Proof of \cref{alt}]
Recall that $\hat x_i$ denotes the word $x$ with the $i$th symbol deleted. We
claim that if $x\in[4]^n$ is a proper coloring,
\begin{equation}\label{qrecur2}Q(x)=\sum_{i=1}^nQ(\hat x_i).\end{equation}
Once this is proved, the result is immediate, since \cref{build-rec} states
that $B$ satisfies the same recurrence, and $Q(\emptyset)=B(\emptyset)=1$ for
the empty word $\emptyset$.

Let $x=\binom{y}{z}$ and let $y$ have $m$ runs. Since $z$ alternates within
each run of $y$, we have $Q(\hat x_i)=0$ whenever $i$ is an interior point of
a run, because $\hat x_i$ is not a proper coloring.  So, we need to compute
$Q(\hat x_i)$ when $i$ is an endpoint of a run of $y$.

Suppose first that $i$ is an endpoint of a run of length at least 2, and
suppose initially that it is not the first or last run. If, for example, $i$
is an endpoint of the $j$th run of $y$, and that run is $\m\m\m\m$, the
relevant part of $x$ is
$$x=\binom yz =
\arraycolsep=1pt
  \left(\begin{array}{rcccl}
    \cdots\ \p\ \p\ \p\   &\quad&    \m\quad\m\quad\m\quad\m     &\quad& \p\ \p\ \p\ \cdots \\
    \hfill\ell_{j-1}       &&     r_{j-1}\hfill\ell_j    &&   r_j
  \end{array}\right),
$$
and if $i$ is the left endpoint of that run, the corresponding $\hat x_i$ is
$$\hat x_i=\binom {\hat y_i}{\hat z_i} =
\arraycolsep=1pt
  \left(\begin{array}{rcccl}
    \cdots\ \p\ \p\ \p\  &\quad&    \ \m\quad\m\quad\m     &\quad& \p\ \p\ \p\ \cdots \\
    \hfill\ell_{j-1}       &&     {-}r_{j-1}\hfill\ell_j    &&   r_j
  \end{array}\right),
$$
while if $i$ is the right endpoint of that run,
$$\hat x_i=\binom {\hat y_i}{\hat z_i} =
\arraycolsep=1pt
  \left(\begin{array}{rcccl}
    \cdots\ \p\ \p\ \p\  &\quad&    \m\quad\m\quad\m \    &\quad& \p\ \p\ \p\ \cdots \\
    \hfill\ell_{j-1}       &&     r_{j-1}\hfill{-}\ell_j    &&   r_j
  \end{array}\right).
$$

In passing from $x$ to $\widehat{x}_i$, the value of $m$ is unchanged, while
the value of $n$ is decreased by $1$. In the first case above, the sign of
$r_{j-1}$ is changed, while in the second case, the sign of $\ell_{j}$ is
changed, and therefore
$$c(w,\hat y_i,\hat z_i)=c(w,y,z)\, (-1)^{\ind[w_{j-1}=-]}$$
in the first case, and
$$c(w,\hat y_i,\hat z_i)=c(w,y,z)\, (-1)^{\ind[w_{j}=+]}$$
in the second. If we set $w_0=w_m=0$ then these also hold when the run is the
first or the last. In both cases, $\hat y_i$ is obtained from $y$ by
shortening the corresponding run by $1$, and $(\hat
y_i)_w=\widehat{(y_w)}_i$. Denote their common value by $\hat y_{w,i}$. So,
the contribution to the right side of \eqref{qrecur2} from (both endpoints
of) this interval is
\begin{equation}\label{contribution1}
2^{n-1-m}\!\!
\sum_{w\in \DD(m-1)}(-1)^{|w|}c(w,y,z)\alpha(\hat y_{w,i})
 \Bigl[(-1)^{\ind[w_{j-1}=-]}+ (-1)^{\ind[w_{j}=+]}\Bigr]
\end{equation}

The last factor $(-1)^{\ind[w_{j-1}=-]}+ (-1)^{\ind[w_{j}=+]}$ can be written
as $2 I(w_{j-1},w_{j})$ where
$$
I(u,v):=\begin{cases}
  +1, &(u,v)=00 \text{ or }{+}{-}; \\
  -1, &(u,v)={-}{+};\\
  0, &\text{otherwise.}
\end{cases}
$$
This follows simply by considering all possibilities for $(w_{j-1},w_j)$,
noting that ${0}{-}$ and ${+}{0}$ are impossible in a dispersed Dyck word.
Therefore \eqref{contribution1} equals
\begin{equation}
\label{contribution2}
 2^{n-m}\sum_{w\in \DD(m-1)}(-1)^{|w|}c(w,y,z)\alpha(\hat y_{w,i})
   I(w_{j-1},w_{j}),
\end{equation}

Now suppose $i$ is the sole element of a run of length $1$. Again $n$ is
decreased by $1$ in passing from $x$ to $\widehat{x}_i$, but now $m$
decreases by $2$ if $1<i<n$, or by $1$ if $i\in\{1,n\}$.  If $i=1$, each
$w'\in \DD(m-2)$ in the sum defining $Q(\widehat x_1)$ can be made into a
$w\in \DD(m-1)$ by adding a $0$ at the beginning, and this gives
$$Q(\hat x_1)=2^{n-m}\sum_{\substack{w\in \DD(m-1):\\w_1=0}}
(-1)^{|w|}c(w,y,z)\alpha(\hat y_{w,1}).$$
 Similarly, for $i=n$, we add a $0$ at the end:
$$Q(\hat x_n)=2^{n-m}\sum_{\substack{w\in \DD(m-1):\\w_{m-1}=0}}
(-1)^{|w|}c(w,y,z)\alpha(\hat y_{w,n}).$$

If $1<i<n$, then (for example)
$$x=\binom yz
=\left(\begin{matrix}
\cdots\ \p\ \p\ \p\ &\m&\p\ \p\ \p\ \cdots\\
 \hfill \ell_{j-1} &\ r_{j-1}=\ell_{j}\ & r_{j}\hfill
\end{matrix}\right),$$
and
$$\hat x_i=\left(\begin{matrix}\hat y_i\\\hat z_i\end{matrix}\right)
=\left(\begin{matrix}
\cdots\ \p\ \p\ \p\ &&\p\ \p\ \p\ \cdots\\
 \hfill \ell_{j-1} && r_{j}\hfill\\
\end{matrix}\right).$$ This is a
proper coloring if and only if $\ell_{j-1}\neq r_{j}$. We will introduce a
factor $(1-\ell_{j-1}r_j)/2$ to account for this constraint. Let $w'\in
\DD(m-3)$ be a word in the sum corresponding to $Q(\widehat x_i)$.  We can
try to make $w'$ into a word in $\DD(m-1)$ by inserting $00$, $\p\m$ or
$\m\p$ before the $(j{-}1)$st symbol of $w'$; denote the resulting words
$w_{00},w_{\p\m},w_{\m\p}$.  Inserting $+-$ introduces an additional factor
$\ell_{j-1} r_{j}$ to $c$, and changes $|w'|$ by $1$.  Exactly one of
$w_{00},w_{\m\p}$ is a dispersed Dyck word (inserting $\m\p$ succeeds
precisely when there is a Dyck word that cannot be broken apart at the
insertion point -- note that e.g.\ $\p\m\p\m$ {\em can} be broken in the
middle, so here we would insert $00$).  Inserting $00$ leaves $c$ and $|w'|$
unchanged, while $\m\p$ multiplies $c$ by $r_{j-1}\ell_j=1$ and changes
$|w'|$ by 1; we introduce an extra sign change in this last case so that we
can get the factor $(1-\ell_{j-1}r_j)/2$.  The conclusion is
\begin{multline*}
(-1)^{|w'|}c(w',\hat y_i,\hat z_i)\,\frac{1-\ell_{j-1} r_{j}}{2}\\
=\tfrac12\!\!\!\!\sum_{\substack{w\in \DD(m-1)\cap \\
\{w_{00},w_{\p\m},w_{\m\p}\}}}\!\! (-1)^{|w|}c(w,y,z)
\,(-1)^{\ind[w=w_{\m\p}]}.
\end{multline*}
Therefore,
\begin{equation}\label{singleton}
Q(\hat x_i)=2^{n-m}\sum_{w\in \DD(m-1)}(-1)^{|w|}c(w,y,z)\alpha(\hat y_{w,i})
I(w_{j-1},w_{j}),
\end{equation}
where $I(w_{j-1},w_{j})$ is precisely the same quantity as defined for the
earlier case, and where the factor $1/2$ has canceled the extra $2$ in
$2^{(n-1)-(m-2)}$. Finally, note that if we again set $w_0=w_m=0$ then
\eqref{singleton} is valid in the cases $i=1,n$ also.

For each $1\leq j\leq m$, write $\widetilde y_{w,j}=\widehat y_{w,i},$ where
$i=i(j)$ is in the $j$th run of $y$.  This is the same for all runs $j$ that
coalesce into a single run when we form $y_w$. Summing over all runs of $y$,
we see that the right side of (\ref{qrecur2}) can be written as
$$2^{n-m}\sum_{w\in \DD(m-1)}(-1)^{|w|}c(w,y,z)\sum_{j=1}^{m}\alpha(\widetilde y_{w,j})
\,I(w_{j-1},w_{j}).$$
 Each Dyck word in $w$ corresponds to a run of $y_w$, as
does each $00$ (where again we take $w_0=w_m=0$). Every Dyck word contains
exactly one more $+-$ than $-+$. Therefore, the sum of $I(w_{j-1},w_j)$ over
those $j$ that correspond to a given run of $y_w$ is 1. By \cref{beta-rec},
the right side of \eqref{qrecur2} agrees with $Q(x)$.
\end{proof}

Finally, we state the promised alternative recurrence for $Q$ that allows for
efficient computation. We have for all proper colorings $x\in[4]^n$,
$$Q(x)=\sum_{r=1}^{n+1} Q_r^0(x),$$
where the quantity $Q_r^k(x)=Q_r^k\binom yz$ is defined for integers $k\geq
0$ and $1\leq r\leq n+1$ by
$$Q_r^k(x) =\ind\bigl[k=0 \text{ and } y_1=(-1)^{r+1}\bigr], \qquad n=1,$$
and for $n\geq 2$,
$$Q_r^k(x)=
\sum_{s\in S}
\begin{cases}
2 Q^k_s(\hat x_1),&y_1=y_2;\\
Q^k_s(\hat x_1) -z_1 Q_s^{k+1}(\hat x_1), & y_1\neq y_2 \text{ and }k=0;\\
z_2 Q_s^{k-1}(\hat x_1) -z_1 Q_s^{k+1}(\hat x_1), & y_1\neq y_2 \text{ and }k>0,
\end{cases}
$$
where
$$S:=\begin{cases}
\{r,\ldots ,n\},& y_1=(-1)^k ; \\
\{1,\ldots ,r-1\},& y_1=(-1)^{k+1}.
\end{cases}
$$
We omit the proof of this, which is a straightforward check given the
following explanation.  The quantity $Q_r^k(x)$ represents an extended
version of $Q(x)$ in which we sum over ``partial dispersed Dyck words" $w$
that can be made into a dispersed Dyck word by appending exactly $k$ $\p$'s
at the beginning, and where in addition each $\alpha(y_w)$ is modified by
restricting to permutations $\pi\in S_{n+1}$ satisfying $\pi_1=r$.

\section{Higher dimensions and shifts of finite type}
\label{extras}

In this section we prove \cref{zd,sft}.  Let $\|\cdot\|=\|\cdot\|_1$ be the
$1$-norm on $\Z^d$.  The distance between two sets $A,B\subseteq\Z^d$ is
$\inf\{\|u-v\|:u\in A,\,v\in B\}$.   We first observe that the definition of
$k$-dependence for graphs given in the introduction is consistent with the
earlier definition for $\Z$. Indeed, suppose $X$ is $k$-dependent according
to the earlier definition.  Then if $(I_j)_{j\in J}$ is any collection of
intervals of $\Z$ no two of which are within distance $k$ then the
restrictions $(X|_{I_j})_{j\in J}$ form an independent family; this follows
by inductively adding one interval at a time.  Now if $A,B\subseteq\Z$ are at
distance greater than $k$ then $X|_A$ and $X|_B$ are independent, since $A$
and $B$ can each be partitioned into subsets that are contained in such a
collection of intervals.

We need the following extension of \cref{main}.  Write $u\simm v$ if
$0<\|u-v\|\leq m$.  A process $(X_v)_{v\in \Z^d}$ is a \df{range-$m$}
\df{$q$-coloring} if each $X_v$ takes values in $[q]$, and almost surely
$X_u\neq X_v$ whenever $u\simm v$.  (A range-$1$ coloring is simply a
coloring).

\begin{cor}\label{range-m}
  Let $d\geq 1$ and $m\geq 1$.  There exists a stationary $m$-dependent range-$m$
  $q$-coloring of $\Z^d$, where $q\leq \exp(c m^d)$ for an
  absolute constant $c$.
\end{cor}

\begin{proof}
A \df{line} is a subset of $\Z^d$ of the form $L=\{a+ih:i\in\Z\}$, where
$a,h\in\Z^d$ and $h\neq 0$.  We call $h$ the \df{direction} of $L$.  We will
place independent copies of the $1$-dependent $4$-coloring along each line in
a suitable family, and combine them to form the desired process.

More precisely, let $H$ be a set comprising exactly one of $h$ and $-h$ for
every $h\in\Z^d$ with $0\simm h$.  (For instance, in the case $m=1$ we can
take $H$ to be the set of $d$ standard basis vectors.)  For each line $L$ of
$\Z^d$ with direction in $H$, take a copy $X^L$ of the $1$-dependent
$4$-coloring of \cref{main}, with the copies being independent for different
lines. Assign the color $X^L_j$ to the point $a+jh$, where
$L=\{a+ih:i\in\Z\}$ (and $a\in L$ is chosen arbitrarily, but is deterministic
and fixed for the particular line). Let $Y^h_v\in[4]$ denote the color thus
assigned to $v\in\Z^d$ by the unique line of direction $h$ passing through
$v$. Finally define $Z_v$ to be the vector $(Y^h_v:h\in H)\in [4]^H$.  The
desired process is $Z=(Z_v)_{v\in\Z^d}$.

Clearly $Z$ is stationary, and its elements take $4^{|H|}$ values.  It is a
range-$m$ coloring since for any $u,v$ with $u \simm v$ there is a line on
which $u,v$ are consecutive points, so $Z_u$ and $Z_v$ differ in the
coordinate corresponding to its direction. (Two points on a line of direction
$h$ are said to be \df{consecutive} on the line if they differ by $\pm h$.)
To check $m$-dependence, note that if $A,B\subseteq \Z^d$ are at distance
greater than $m$ from each other then every line with direction in $H$ that
intersects both $A$ and $B$ does so in two non-consecutive sets. Thus $Z|_A$
and $Z|_B$ are functions of independent collections of random variables.
\end{proof}

\begin{proof}[Proof of \cref{zd}(i)]
This is \cref{range-m} with $m=1$.  (The number of colors
is $q=4^d$).
\end{proof}

To state the relevant results from \cite{hsw} we need to
generalize block-factors to $d$ dimensions. Denote the ball
$B(r):=\{v\in\Z^d:\|v\|\leq r\}$.  A \df{block-factor map}
is a map $F:\R^{\Z^d}\to\R^{\Z^d}$ characterized by an
integer $r$ called the \df{radius} and a measurable
function $f:\R^{B(r)} \to \R$ via
$$(F(x))_v = f\bigl((\theta^{-v} x)|_{B(r)}\bigr), \qquad x\in
\R^{\Z^d},\, v\in\Z^d,$$
 where $\theta^{-v}$ denotes translation by
 $-v$ (so $(\theta^{-v}x)_u=x_{v+u}$).
(Thus, an $r$-block-factor on $\Z$ is a process that can be
expressed as a radius-$\lfloor r/2\rfloor$ block-factor map
of an i.i.d.\ process on $\Z$).

\begin{lemma}\label{block-k-dep}
Let $X$ be a stationary $k$-dependent process on $\Z^d$ and let $F$ be a
radius-$r$ block-factor map.  Then $F(X)$ is stationary and
$(2r+k)$-dependent.
\end{lemma}

\begin{proof}
This follows easily from the definitions.
\end{proof}

\begin{samepage}
\begin{thm}[Holroyd, Schramm and Wilson; \cite{hsw}]\label{from-hsw}\
\begin{ilist}
\item Let $d\geq 1$.  There exists $m$ such that for any $q$ there exists
    a block-factor map $F$ with the following property.  If
$X$ is a range-$m$ $q$-coloring of $\Z^d$ then $F(X)$ is
a (range-$1$) $4$-coloring of $\Z^d$.
\item Let $S$ be a non-lattice shift of finite type on
    $\Z$. There exists
    $m$ such that for any $q$ there exists a block-factor map $F$ with
    the following property.  If $X$ is a range-$m$ $q$-coloring of $\Z$
    then $F(X)$ belongs to $S$ almost surely.
\end{ilist}
\end{thm}
\end{samepage}

The somewhat awkward series of quantifiers above reflects the need to
encapsulate the relevant results from \cite{hsw} cleanly without going into
details of their proofs.

\begin{proof}[Proof of \cref{zd}(ii) and \cref{sft}]
The results are immediate from
\cref{range-m,block-k-dep,from-hsw}.
\end{proof}

We make a few remarks about the scope of \cref{zd,sft}.
While the colorings of \cref{zd} are stationary (meaning
invariant under translations), they are not invariant in
law under all {\em isometries} of $\Z^d$, because the proof
imposed an ordering on the set of line directions, which is
not invariant under permuting the coordinates.  We do not
know how to construct an isometry-invariant
finitely-dependent coloring of $\Z^d$ for $d \geq 2$.
Similar remarks apply to trees, as pointed out by Russell
Lyons (personal communication).  Treating a regular tree as
the Cayley graph of a free group, we obtain a $1$-dependent
coloring that is invariant under the action of the group
itself (which is vertex-transitive), by the same approach
as in the proof of \cref{zd}. However, we do not know how
to construct a fully automorphism-invariant finitely
dependent coloring.

As remarked in the introduction, another result of
\cite{hsw} implies that there is no stationary
$k$-dependent $3$-coloring of $\Z^d$ for any $k$ and $d\geq
2$.  In fact, there is no stationary $3$-coloring of $\Z^2$
whose correlations decay faster than a certain polynomial
rate.

It is straightforward to check that if $S$ is a {\em
lattice} shift of finite type on $\Z$ then there is no
stationary finitely dependent process that belongs almost
surely to $S$. In fact, there is no stationary {\em mixing}
process that belongs to $S$; again, details appear in
\cite{hsw}.

\section{One-dependent hard-core processes}
\label{sec-hardcore}

In this section we prove \cref{bounds}.  We also discuss
properties of $1$-dependent hard-core processes, which are
interesting in their own right. Let $G=(V,E)$ be a simple,
countable, undirected graph with all degrees finite. Recall
that a hard-core process $J=(J_v)_{v\in V}$ is a
$\{0,1\}$-valued process with no adjacent $1$'s, and that
$\ph(G)$ is defined to be the supremum of $p$ for which
there exists a $1$-dependent hard-core process with all its
one-vertex marginals $\P(J_v=1)$ equal to $p$.

In \cref{ph,physics} below we record some simple but interesting observations
about $\ph$.  Closely related ideas appear in work of Scott and Sokal
\cite{scott-sokal,scott-sokal2}, where a rich web of interconnections
involving mathematical physics and probabilistic combinatorics is explored.
The arguments we use in the proofs of \cref{ph,physics} are largely present
in those articles, at least implicitly.  However, our particular viewpoint
(focussing on $1$-dependent hard-core processes, especially on infinite
graphs) is apparently novel, as is our application to coloring. As another
application of our approach, we give an alternative proof of a result of
Shearer \cite{shearer} at the end of this section.

\begin{lemma}\label{ph} Let $G$ be a graph.
For each $p\leq \ph$ there exists a unique $1$-dependent
hard-core process with all one-vertex marginals equal to
$p$.  This process is invariant in law under all
automorphisms of $G$.
\end{lemma}

\begin{proof}
We first observe a general monotonicity statement: if a $1$-dependent
hard-core process $J$ with one-vertex marginals $\P(J_v=1)=p_v$ exists, and
if $p'_v\leq p_v$ for all $v\in V$, then such a process exists with marginals
$(p'_v)$.  This follows by thinning: let $(\epsilon_v)_{v\in V}$ be
$\{0,1\}$-valued, independent of each other and of $J$, with
$\P(\epsilon_v=1)=p'_v/p_v$; then take $J'_v=\epsilon_v J_v$.

The above shows that a $1$-dependent hard-core process
exists for all $p<\ph$.  To extend this to $p=\ph$, take a
sequence $p_n\nearrow \ph$ and a process for each $p_n$,
and consider a subsequential weak limit in distribution $J$
(which exists, by compactness). Since probabilities of all
cylinder events converge, $J$ has all marginals equal to
$\ph$, and is a $1$-dependent hard-core process.

Uniqueness and automorphism-invariance follow from the more general fact that
the law of a $1$-dependent hard-core process is determined by its one-vertex
marginals $p_v=\P(J_v=1)$. Indeed, the law of a binary process $J$ is
determined by the probabilities $\P(J \equiv 1 \text{ on }A)$ for finite
$A\subseteq V$, since all other cylinder probabilities can be computed from
them by inclusion-exclusion. But this probability equals $0$ if $A$ contains
two neighbors, and otherwise it is $\prod_{v\in A} p_v$.
\end{proof}

For a finite set of vertices $A\subseteq V$ and $\lambda\in\R$, define
$$Z_A(\lambda):=\sum_{B\in\mathcal{I}(A)} \lambda^{|B|},$$
where $\mathcal{I}$ is the set of all independent subsets
of $A$ (or hard-core configurations), i.e.\ subsets of $A$
that do not contain any two neighbors in $G$. This is the
partition function of the standard hard-core model of
statistical physics; it is also known as the independence
polynomial of the induced subgraph of $A$. See e.g.\
\cite{levit-mandrescu,scott-sokal}.

\begin{lemma}\label{physics}
Let $G=(V,E)$ be a graph and let $p\in[0,1]$.  We have
$p\leq p_h$ if and only if $Z_A(-p)\geq 0$ for every finite
$A\subseteq V$.  If $G$ is infinite and connected then this
is also equivalent to the statement that the strict
inequality $Z_A(-p)> 0$ holds for every finite $A\subset
V$.
\end{lemma}

\begin{proof}
Suppose that $p\leq \ph$, so a $1$-dependent hard-core
process $J$ with marginals $p$ exists.  Then by
inclusion-exclusion,
\begin{equation}\label{inc-exc}
\P(J\equiv 0 \text{ on }A) = \sum_{B\subseteq A} (-1)^{|B|} \, \P(J \equiv 1
\text{ on }B) = Z_A(-p),
\end{equation}
so the last quantity is non-negative.

Moreover, all other cylinder probabilities can be expressed in terms of those
above. Let $B,C$ be disjoint finite sets of vertices with $B\in
\mathcal{I}(V)$, and let $C'$ be the set of vertices of $C$ that have no
neighbor in $B$. Then
\begin{align*}
\P(J\equiv 1 \text{ on }B, \;J\equiv 0 \text{ on }C)
&=\P(J\equiv 1 \text{ on }B, \;J\equiv 0 \text{ on }C') \\
&=p^{|B|} \, Z_{C'}(-p).
\end{align*}
Thus, given $Z_A(-p)\geq 0$ for all $A$, we can compute
non-negative expressions for all cylinder probabilities,
and it is easy to check that they are consistent and give
rise to a $1$-dependent hard-core process with marginals
$p$.  Thus $p\leq \ph$.

Here is a useful recurrence.  Suppose $A\subseteq V$ is finite, let $u\in A$,
and define $A':=A\setminus\{u\}$ and $A'':=A'\setminus N(u)$, where $N(u)$
denotes the set of neighbors of $u$. Then by an argument similar to the
above,
\begin{equation}\label{recur}
  Z_A(-p)=
 Z_{A'}(-p)
 -p\,Z_{A''}(-p).
\end{equation}
(Indeed, it is a standard and straighforward fact that this
identity holds for any parameter $\lambda$, regardless of
the existence of the process $J$; see e.g.\
\cite{levit-mandrescu,scott-sokal}).

To prove the final claimed equivalence, suppose that $G$ is
infinite and connected.  Let $0<p\leq \ph$. (If $\ph=0$
then the claim is trivial.) Suppose that $Z_A(-p)=0$ for
some finite $A\subset V$, and let $A$ be minimal with this
property.  There exists a vertex $u\notin A$ that is
adjacent to $A$.  Let $B=A\cup\{u\}$,  $B'=A$, and
$B''=A\setminus N(u)$.  Then applying \eqref{recur} to
$B,B',B''$ gives that $Z_B(-p)$ is negative, a
contradiction.
\end{proof}

For an infinite connected $G$, our critical point $\ph$
coincides with the critical point $\lambda_{\rm c}$ defined
in \cite{scott-sokal} (in (5.3) and the immediately
following remark) in terms of the complex zeros of $Z$.
This follows immediately from \cref{physics} above together
with Theorem 2.2(b,c) and (3.1) of \cite{scott-sokal}.

Consequently, the following bounds on $\ph$ are available. For any infinite
connected graph $G$ of maximum degree $\Delta$,
\begin{equation}\label{graph-bounds}
\frac{(\Delta-1)^{\Delta-1}}{\Delta^\Delta}
\leq \ph(G) \leq \frac14,\qquad \Delta\geq 2.
\end{equation}
For the infinite $\Delta$-regular tree $T_{\Delta}$, the
lower bound is sharp:
\begin{equation}\label{tree}
\ph(T_{\Delta}) = \frac{(\Delta-1)^{\Delta-1}}{\Delta^\Delta},
\qquad \Delta\geq 2.
\end{equation}
For the hypercubic lattice $\Z^d$,
\begin{equation}\label{zd-bounds}
\frac{(2d-1)^{2d-1}}{(2d)^{2d}}
\leq \ph(\Z^d) \leq \frac{d^d}{(d+1)^{d+1}}, \qquad d\geq 1.
\end{equation}
Proofs of \eqref{graph-bounds},\eqref{zd-bounds} appear in \cite[\S5.2,
\S8.4]{scott-sokal2}; the lower bound in \eqref{graph-bounds} amounts to the
Lov\'asz local lemma.  The equality \eqref{tree} is proved in \cite{shearer},
and an exposition of the proof also appears in
\cite{scott-sokal,scott-sokal2}.
 Note that $\ph(\Z)=1/4$.  This is a special
case of all of \eqref{graph-bounds},\eqref{tree},\eqref{zd-bounds}, and also
follows from the proof of \cref{quarter}. Using rigorous computer-assisted
methods, we supply the following improvement on \eqref{zd-bounds} in
dimensions $2$ and $3$.
\begin{lemma}\label{computer} We have the strict
inequalities
 $$\ph(\Z^2) <\frac18; \qquad \ph(\Z^3) <\frac{1}{11}.$$
\end{lemma}

\begin{proof}
The recursion \eqref{recur} gives $Z_A(-p)$ in terms of
$Z_B(-p)$ for smaller sets $B\subset A$.  We use this to
compute $Z_A(-p)$ numerically for rectangular boxes of the
form $A=[a]\times[b]\subset\Z^2$ and
$A=[a]\times[b]\times[c]\subset\Z^3$. After some
experimentation to find appropriate box sizes, we obtained
$$Z_{[13]\times [10]}(-1/8)<0; \qquad Z_{[12]\times[4]\times[4]}(-1/11)<0,$$
giving the claimed bounds.

One must choose which vertex $u$ to remove from a set $A$ when applying
\eqref{recur}.  We always chose the lexicographically largest $u\in A$, as
this tends to limit the number of smaller sets that need to be considered.
The method turns out to be numerically unstable, so that floating-point
arithmetic cannot be used.  Instead we used exact arbitrary-precision
rational arithmetic. The quantity $Z_{[12]\times[4]\times[4]}(-1/11)$ is a
fraction with $100$ digits in the denominator, and required the computation
of $Z_B(-1/11)$ for $89077$ sets $B\subseteq[12]\times[4]\times[4]$. (We
provide the computer code in an appendix to the arxiv version of this paper.)
\end{proof}

\begin{proof}[Proof of \cref{bounds}]
As remarked in the introduction, the existence of a $1$-dependent
$q$-coloring $X$ with the variables $(X_v)_{v\in V}$ identically distributed
implies that $q\geq 1/\ph$. Indeed, let $a\in [q]$ be a color with the
largest marginal probability $p_a$ ($\geq 1/q$);  then $J_v:=\ind[X_v=a]$
defines a $1$-dependent hard-core process, so $p_a\leq \ph$. Now use the
upper bounds in \eqref{tree}, \eqref{zd-bounds} and \cref{computer}.
\end{proof}

The (non-rigorous) estimate $\ph(\Z^2)=0.11933888188(1)$ was computed in
\cite{todo}.  That this is greater than $1/9$ indicates that a $9$-coloring
of $\Z^2$ cannot be ruled out by the methods of this section.

Finally, we present an application of our approach in the
context of \cite{scott-sokal}.  Motivated by the case of
$\Z$ in \cref{properties}(ii), we give a very simple
explicit construction of the critical $1$-dependent
hard-core process $J$ on the $\Delta$-regular tree
$T_{\Delta}$, thus providing an alternative proof of the
upper bound on $\ph(T_{\Delta})$ in \eqref{tree}. (The
original proof in \cite{shearer} used analytic methods).
Fix an end of the tree.  Assign the vertices i.i.d.\
$\{0,1\}$-valued labels that are $1$ with probability
$1/\Delta$.  Then let $J_v$ equal $1$ if and only if $v$
has label $1$ and all its children have label $0$.  (The
\df{children} of a vertex are the $\Delta-1$ neighbors that
do not lie on the unique path to the nominated end.) Then
$\P(J_v=1) = (\Delta-1)^{\Delta-1}/\Delta^\Delta$ as
required.  It is interesting that the construction is
invariant only under automorphisms that fix the given end,
while the process itself is fully automorphism-invariant,
by \cref{ph}.  Can the critical process on $T_{\Delta}$ be
expressed as a fully automorphism-equivariant block-factor
of an i.i.d.\ process?

\enlargethispage*{1cm}
\section*{Open Problems}

\begin{ilist}
\item Is the stationary $1$-dependent $4$-coloring of $\Z$ unique? We
    conjecture that the answer is yes. Is the stationary $2$-dependent
    $3$-coloring unique?
\item Is there a finitely dependent coloring
    $(X_i)_{i\in\Z}$
    such that
    $X_i=f(M_i)$ for a stationary {\em countable}-state Markov chain $M$?
    (A finite state space is impossible, while an uncountable one places
    no restriction on the process). Can our two examples be expressed in
    this way?
\item What is the largest possible one-vertex marginal of
    a stationary $k$-dependent hard-core process on $\Z$ for $k\geq 2$?
      Is it $1/3$ when $k=2$?  Is the critical process unique?
\item Can one of our two colorings of $\Z$ be expressed as a block-factor
    of the other? As a finitary factor?
\item Is there a stationary finitely dependent coloring
    of
    $\Z$ that can
    be expressed as a finitary factor of an i.i.d.\ process with
    finite mean coding radius?  (In \cite{h-finfin}, the
    $4$-coloring is expressed as a finitary factor with
    infinite mean coding radius.)
\item What is the minimum number of colors $q$ needed for a stationary
    $1$-dependent $q$-coloring of $\Z^d$, for each $d\geq 2$?  (For
    $\Z^2$, the answer is between $9$ and $16$).
\item Does there exist a finitely dependent coloring of
    $\Z^d$ for $d\geq 2$ that it is invariant in law under
    all isometries of $\Z^d$?  Does there exist a finitely
    dependent coloring of a regular tree that is invariant
    under all automorphisms, or all
    automorphisms that fix a given end?
\item  On which transitive graphs is the existence of a $1$-dependent
    hard-core process with all one-vertex marginals equal to $1/q$
    sufficient for the existence of an automorphism-invariant
    $1$-dependent $q$-coloring?  (It is necessary on any graph, and
    sufficient on $\Z$).
\end{ilist}

\section*{Acknowledgements}

We thank Itai Benjamini, David Brydges, Ronen Eldan,
Jeong-Han Kim, Russell Lyons, Ben Morris, Fedja Nazarov,
Robin Pemantle, Benjamin Weiss, Peter Winkler and Fuxi
Zhang for valuable discussions.

\bibliographystyle{habbrv}
\bibliography{col}

\newpage
\section*{Appendix: computer code}

Below we give the Python 2.7 code used in the proof of
\cref{computer}.  It computes the following values of the
independence polynomial for rectangular grids.  (The first
is included as a check).
\begin{gather*}
Z_{[3]\times[3]}(-1/5)=-\tfrac{21}{3125};\\
Z_{[13]\times[10]}(-1/8)=-\tfrac{60294169567161237625416728069877775945051113}
{25108406941546723055343157692830665664409421777856138051584};
\\
Z_{[12]\times[4]\times[4]}(-1/11)= -\;\frac{
\begin{smallmatrix}
 46344295466778955212216048923\\
 88528097877566844283627882753\\
 10889047735211360981028087687
\end{smallmatrix}
}{
\begin{smallmatrix}
941234365126854052600118651191150\\
657486806311046954882395087600037\\
9062365652829504091329792873336961
\end{smallmatrix}
}.
\end{gather*}

\small
\begin{verbatim}
from fractions import Fraction

def Z(A,t):     # independence polynomial of set A
    if A:
        if (A,t) not in memo:   # if not already computed
            u=max(A)            # choose site to remove
            B=A.difference([u])
            C=B.difference(nbrs(u))
            memo[(A,t)]=Z(B,t)+t*Z(C,t)
        return memo[(A,t)]
    else:
        return 1                # empty set

def nbrs(u):    # neighbors of a site in Z^d
    for i in xrange(len(u)):
        for k in -1,1:
            yield u[:i]+(u[i]+k,)+u[i+1:]

def grid(s):    # rectangular box in Z^d
    if s:
        return frozenset((i,)+u for u in grid(s[1:])
                                for i in xrange(s[0]))
    else:
        return frozenset([()])

memo={}
print Z(grid((3,3)),Fraction(-1,5))
print Z(grid((13,10)),Fraction(-1,8))
print Z(grid((12,4,4)),Fraction(-1,11))
\end{verbatim}
\normalsize


\end{document}